\documentclass[12pt,reqno]{amsart}
\usepackage{amsmath,amssymb,amsfonts,amscd,latexsym,amsthm,mathrsfs,verbatim,comment,cite}
\usepackage{hyperref}
\newtheorem{theorem}{Theorem}[section]
\newtheorem{conjecture}[theorem]{Conjecture}
\newtheorem{lemma}[theorem]{Lemma}

\let\le\leqslant
\let\ge\geqslant

\renewcommand{\pmod}[1]{\;\allowbreak(\operatorname{mod}#1)}

\numberwithin{equation}{section}

\begin{document}

\title{Dwork-type supercongruences through~a~creative~$q$-microscope}

\date{7 January 2020. \emph{Revised}: 23 November 2020}

\author{Victor J. W. Guo}
\address{School of Mathematics and Statistics, Huaiyin Normal University, Huai'an 223300, Jiangsu, People's Republic of China}
\email{jwguo@hytc.edu.cn}

\author{Wadim Zudilin}
\address{Department of Mathematics, IMAPP, Radboud University, PO Box 9010, 6500~GL Nijmegen, Netherlands}
\email{w.zudilin@math.ru.nl}

\thanks{The first author was partially supported by the National Natural Science Foundation of China (grant 11771175).}


\begin{abstract}
We develop an analytical method to prove congruences of the type
$$
\sum_{k=0}^{(p^r-1)/d}A_kz^k \equiv \omega(z)\sum_{k=0}^{(p^{r-1}-1)/d}A_kz^{pk} \pmod{p^{mr}\mathbb Z_p[[z]]}
\quad \text{for}\; r=1,2,\dots,
$$
for primes $p>2$ and fixed integers $m,d\ge1$, where $f(z)=\sum_{k=0}^\infty A_kz^k$ is an `arithmetic' hypergeometric series.
Such congruences for $m=d=1$ were introduced by Dwork in 1969 as a tool for $p$-adic analytical continuation of $f(z)$.
Our proofs of several Dwork-type congruences corresponding to $m\ge2$ (in other words, supercongruences) are based on constructing and proving their suitable $q$-analogues, which in turn have their own right for existence and potential for a $q$-deformation of modular forms and of cohomology groups of algebraic varieties.
Our method follows the principles of creative microscoping introduced by us to tackle $r=1$ instances of such congruences;
it is the first method capable of establishing the supercongruences of this type for general~$r$.
\end{abstract}

\maketitle

\section{Introduction}
\label{sec1}


Extending his work on the rationality of the zeta function of an algebraic variety defined over a finite field, Dwork \cite{Dw69} considered a question of continuing analytical solutions $f(z)=\sum_{k=0}^\infty A_kz^k$ of linear differential equations $p$-adically.
A general strategy was to verify that the truncated sums $f_r(z)=\sum_{k=0}^{p^r-1}A_kz^k$, where $r=0,\allowbreak1,2,\dots$, satisfy the so-called Dwork congruences \cite{MV16}
\begin{equation}
\frac{f_{r+1}(z)}{f_r(z^p)}
\equiv\frac{f_r(z)}{f_{r-1}(z^p)}\pmod{p^r\mathbb Z_p[[z]]}
\quad \text{for}\; r=1,2,\dots
\label{eq1}
\end{equation}
(see \cite[Theorem~3]{Dw69} for a precise statement).
Formally, one needs the condition $f_1(z^p)=\sum_{k=0}^{p-1}A_kz^{pk}\not\equiv0\pmod{p\mathbb Z_p[[z]]}$ to make sense of \eqref{eq1}.
Then the congruences imply the existence of a $p$-adic analytical function (`unit root') $\omega(z)$ such that
$$
\omega(z)=\lim_{r\to\infty}\frac{f_r(z)}{f_{r-1}(z^p)};
$$
in other words,
$$
\omega(z)\equiv\frac{f_r(z)}{f_{r-1}(z^p)}\pmod{p^r\mathbb Z_p[[z]]}
\quad \text{for}\; r=1,2,\dotsc.
$$
Notice that the argument extends to the cases when $f_1(z^p)\equiv0\pmod{p\mathbb Z_p[[z]]}$ but $f_1(z^p)\not\equiv0\pmod{p^m\mathbb Z_p[[z]]}$ for some $m\ge2$, provided the congruences \eqref{eq1} hold modulo a higher power of~$p$, for example,
\begin{equation}
\frac{f_{r+1}(z)}{f_r(z^p)}
\equiv\frac{f_r(z)}{f_{r-1}(z^p)}\pmod{p^{mr}\mathbb Z_p[[z]]}
\quad \text{for}\; r=1,2,\dotsc.
\label{eq1a}
\end{equation}
It is this type of congruences that we refer to as Dwork-type supercongruences; other truncations of the initial power series are possible as well, usually of the type $f_r(z)=\sum_{k=0}^{(p^r-1)/d}A_kz^k$
for some fixed positive integer~$d$.
Whether the congruences \eqref{eq1a} are `super' ($m\ge2$) or not ($m=1$), we conclude from them that
\begin{equation}
f_r(z)\equiv \omega(z)f_{r-1}(z^p)\pmod{p^{mr}\mathbb Z_p[[z]]}
\quad \text{for}\; r=1,2,\dotsc.
\label{eq2}
\end{equation}
This gives an equivalent\,---\,somewhat more transparent\,---\,way to state Dwork-type (super)congruences in the case of known unit root $\omega(z)$.

Our illustrative examples include
\begin{align}
\sum_{k=0}^{(p^r-1)/2}(8k+1)\frac{\binom{4k}{2k}{\binom{2k}{k}}^2}{2^{8k}3^{2k}}
&\equiv p\biggl(\frac{-3}p\biggr)
\sum_{k=0}^{(p^{r-1}-1)/2}(8k+1)\frac{\binom{4k}{2k}{\binom{2k}{k}}^2}{2^{8k}3^{2k}} \pmod{p^{3r}}, \label{ram1a-r}  \\
\sum_{k=0}^{p^r-1}(8k+1)\frac{\binom{4k}{2k}{\binom{2k}{k}}^2}{2^{8k}3^{2k}}
&\equiv p\biggl(\frac{-3}p\biggr)
\sum_{k=0}^{p^{r-1}-1}(8k+1)\frac{\binom{4k}{2k}{\binom{2k}{k}}^2}{2^{8k}3^{2k}} \pmod{p^{3r}}, \label{ram1b-r}
\end{align}
where $\bigl(\frac{-3}{\cdot}\bigr)$ denotes the Kronecker symbol,
valid for any prime $p>3$ and integer $r\ge1$ and corresponding to the
truncation of the power series
$$
\sum_{k=0}^{\infty}(8k+1)\binom{4k}{2k}{\binom{2k}{k}}^2\frac{z^k}{2^{8k}3^{2k}}
$$
at $z=1$.
We point out that not so many supercongruences of this type are recorded in the literature; the principal sources are the conjectures from Swisher's paper \cite{Sw15}, in turn built on Van Hamme's list \cite{VH97},
and a geometric heuristics for hypergeometric series $f(z)$ outlined by Roberts and Rodriguez-Villegas in~\cite{RRV19}.
The only \emph{proven} cases known (namely, weaker forms of Conjectures (C.3) and (J.3) from \cite{Sw15} together with their companions) for arbitrary $r\ge1$ are due to the first author \cite{Gu20d}.

The principal goal of this paper is to extend the approach of \cite{Gu20d} and establish general techniques for proving Dwork-type supercongruences
using the method of creative microscoping, which we initiated in~\cite{GZ19a} for proving $r=1$ instances of such supercongruences.
Observe that such $r=1$ cases of \eqref{ram1a-r}, \eqref{ram1b-r} (known as Ramanujan-type supercongruences \cite{Zu09}) served as principal illustrations of how the creative microscope machinery works.
It should be therefore not surprising that we place them again as principal targets.
Here we prove Dwork-type supercongruences \eqref{ram1a-r}, \eqref{ram1b-r} by establishing the following $q$-analogues of them.

\begin{theorem}
\label{main-1}
Let $n>1$ be an integer coprime with $6$ and let $r\ge 1$. Then, modulo $[n^r]\prod_{j=1}^r\Phi_{n^j}(q)^2$,
\begin{align}
&\sum_{k=0}^{(n^r-1)/2}[8k+1]\frac{(q;q^2)_k^2 (q;q^2)_{2k}}{(q^6;q^6)_k^2 (q^2;q^2)_{2k}}q^{2k^2} \notag\\
&\qquad\equiv
q^{(1-n)/2}[n]\biggl(\frac{-3}{n}\biggr)
\sum_{k=0}^{(n^{r-1}-1)/2}[8k+1]_{q^n}\frac{(q^n;q^{2n})_k^2 (q^n;q^{2n})_{2k}}{(q^{6n};q^{6n})_k^2 (q^{2n};q^{2n})_{2k}}q^{2nk^2},
\label{q4a}
\displaybreak[2]\\
&\sum_{k=0}^{n^r-1}[8k+1]\frac{(q;q^2)_k^2 (q;q^2)_{2k}}{(q^6;q^6)_k^2 (q^2;q^2)_{2k}}q^{2k^2}  \notag\\
&\qquad\equiv
q^{(1-n)/2}[n]\biggl(\frac{-3}{n}\biggr)
\sum_{k=0}^{n^{r-1}-1}[8k+1]_{q^n}\frac{(q^n;q^{2n})_k^2 (q^n;q^{2n})_{2k}}{(q^{6n};q^{6n})_k^2 (q^{2n};q^{2n})_{2k}}q^{2nk^2}.
\label{q4b}
\end{align}
\end{theorem}

Here and throughout the paper we adopt the standard $q$-notation:
$(a;q)_n=(1-a)(1-aq)\dotsb(1-aq^{n-1})$ is the $q$-shifted factorial ($q$-Pochhammer symbol),
$[n]=[n]_q=(1-q^n)/(1-q)$ is the $q$-integer,
and
\begin{align*}
\Phi_n(q)=\prod_{\substack{1\le k\le n\\ \gcd(n,k)=1}}(q-\zeta_n^k),
\end{align*}
is the $n$-th cyclotomic polynomial, where $\zeta_n=e^{2\pi i/n}$ is an $n$-th primitive root of unity.
Also recall the ordinary shifted factorial $(a)_n=\Gamma(a+n)/\Gamma(a)=a(a+\nobreak1)\allowbreak\dotsb(a+n-1)$ for $n=0,1,2,\dots$\,.
In what follows, the congruence $A_1(q)/A_2(q)\equiv0\pmod{P(q)}$
for polynomials $A_1(q),A_2(q),P(q)\in\mathbb Z[q]$ is understood as $P(q)$ divides $A_1(q)$ and is coprime with $A_2(q)$;
more generally, $A(q)\equiv B(q)\pmod{P(q)}$ for rational functions $A(q),B(q)\in\mathbb Z(q)$ means $A(q)-B(q)\equiv0\pmod{P(q)}$.

It is not hard to check (see \cite{GZ19a,Zu20} for related details of this computation) that, when $n=p$ is a prime and $q\to 1$, the $q$-supercongruences \eqref{q4a} and \eqref{q4b} reduce to \eqref{ram1a-r} and \eqref{ram1b-r}, respectively.

Another family of Dwork-type supercongruences
\begin{align}
\sum_{k=0}^{(p^r-1)/2} \frac{(\frac{1}{2})_k^3}{k!^3}(3k+1)2^{2k}
&\equiv p \sum_{k=0}^{(p^{r-1}-1)/2} \frac{(\frac{1}{2})_k^3}{k!^3}(3k+1)2^{2k} \pmod{p^{3r}}, \label{3k+1-a}\\
\sum_{k=0}^{p^r-1} \frac{(\frac{1}{2})_k^3}{k!^3}(3k+1)2^{2k}
&\equiv p \sum_{k=0}^{p^{r-1}-1} \frac{(\frac{1}{2})_k^3}{k!^3}(3k+1)2^{2k} \pmod{p^{4r-\delta_{p,3}}},  \label{3k+1-b}
\end{align}
expectedly valid for any prime $p>2$ and integer $r\ge1$, originate from the \emph{divergent} hypergeometric series
$$
\sum_{k=0}^\infty \frac{(\frac{1}{2})_k^3}{k!^3}(3k+1)(2^2z)^k
$$
at $z=1$. (Here $\delta_{i,j}$ is the usual Kronecker delta, $\delta_{i,j}=1$ if $i=j$ and $\delta_{i,j}=0$ otherwise.)
The congruences \eqref{3k+1-a} and \eqref{3k+1-b} modulo $p^3$ merge into the single entry
\begin{align}
\sum_{k=0}^{(p-1)/2} \frac{(\frac{1}{2})_k^3}{k!^3}(3k+1)2^{2k} \equiv p \pmod{p^3}
\quad\text{for $p>2$},  \label{eq:div-1}
\end{align}
when $r=1$, because $(\frac12)_k\equiv0\pmod p$ for $(p-1)/2<k\le p-1$; these `divergent' Ramanujan-type supercongruences were proved by Guillera and the second author \cite{GZ12} (while independently observed numerically by Sun \cite[Conjecture 5.1\,(ii)]{Su11}).
The first author \cite{Gu20a} gave a $q$-analogue of \eqref{eq:div-1} and recorded \eqref{3k+1-a}, \eqref{3k+1-b} as
conjectures. In this paper we prove the supercongruences \eqref{3k+1-a}, \eqref{3k+1-b} modulo $p^{3r}$ by establishing the following $q$-counterparts.

\begin{theorem}
\label{main-2}
Let $n>1$ be odd and $r\ge 1$. Then, modulo $[n^r]\prod_{j=1}^r\Phi_{n^j}(q)^2$,
\begin{align}
\sum_{k=0}^{(n^r-1)/2}[3k+1]\frac{(q;q^2)_k^3 q^{-{k+1\choose 2} } }{(q;q)_k^2 (q^2;q^2)_k}
&\equiv q^{(1-n)/2}[n]\sum_{k=0}^{(n^{r-1}-1)/2}[3k+1]_{q^n}\frac{(q^n;q^{2n})_k^3 q^{-n{k+1\choose 2} } }{(q^n;q^n)_k^2 (q^{2n};q^{2n})_k}, \label{eq:q-div-WZ-1} \\
\sum_{k=0}^{n^r-1}[3k+1]\frac{(q;q^2)_k^3 q^{-{k+1\choose 2} } }{(q;q)_k^2 (q^2;q^2)_k}
&\equiv q^{(1-n)/2}[n]\sum_{k=0}^{n^{r-1}-1}[3k+1]_{q^n}\frac{(q^n;q^{2n})_k^3 q^{-n{k+1\choose 2} } }{(q^n;q^n)_k^2 (q^{2n};q^{2n})_k}.  \label{eq:q-div-WZ-2}
\end{align}
\end{theorem}

Although $q$-supercongruences serve here as a principal tool for proving their non-$q$-counterparts,
they have established themselves as an independent topic.
For some recent developments on $q$-supercongruences we refer the reader to the papers
\cite{Go19,Gu18a,Gu18b,Gu19c,Gu20a,Gu20b,Guo-mod4,Guo-rima,GPZ17,GS20a,GS20b,GS20c,GZ19a,GZ19b,LW,NP18,St19,Ta13,Zu19}.

Both hypergeometric identities and congruences for their truncations originate from their $q$-hypergeometric versions in a very natural way, through the asymptotics as $q\to1$ for the former and as $q$ approaches other roots of unity for the latter;
it is this asymptotic analysis at roots of unity, which we refer to as `$q$-microscopic'.
Notice that proving a congruence $A(q)\equiv B(q)\pmod{\Phi_N(q)}$ is equivalent to verifying that $A(\zeta)=B(\zeta)$ for all primitive $N$-th roots of unity~$\zeta$.
Furthermore, proofs of the congruences require `creative' introduction of extra (generic) parameter $a$ (and, possibly, some other);
those parameters are often (but not always!) suggested by general forms of the underlying $q$-hypergeometric identities.
The intermediate parametric supercongruences of the form $A(q,a)\equiv B(q,a)$ are verified to be true modulo polynomials $a-q^N$ and $1-aq^N$ (for particular choices of integers~$N$) by showing that $A(q,q^N)=B(q,q^N)$ and $A(q,q^{-N})=B(q,q^{-N})$; afterwards, the dependence on the parameter is eliminated via a careful analysis of degeneration as $a\to1$. A plain overview of the method can be found in~\cite{Zu20}.
Quite remarkably, the strategy of creative $q$-microscoping makes it possible to prove many congruences that are not accessible to other techniques.

The exposition below is organized as follows.
In Section~\ref{sec2} we provide detailed proofs of Theorems~\ref{main-1} and~\ref{main-2}.
The methodology set up in that section is further used in Section~\ref{sec3} to prove several other $q$-supercongruences whose limiting $q\to1$ cases correspond to Dwork-type supercongruences, occasionally conjectured in the existing literature.
Most of the results in Section~\ref{sec3} are supplied with sketches of their proofs.
Finally, in Section~\ref{sec4} we leave several open problems about $q$-congruences behind Dwork-type (super)congruences~\eqref{eq2} and discuss possible future of the $q$-setup.

In our proofs below we make use of transformation formulas of basic hypergeometric series \cite{GR04}
$$
_{s+1}\phi_{s}\biggl[\begin{matrix}
a_0, \, a_1, \, \, \dots, \, a_s \\[2.5pt]
b_1, \, b_2, \, \dots, \, b_s
\end{matrix}; q,\, z \biggr]
=\sum_{k=0}^{\infty}\frac{(a_0,a_1,\dots,a_s;q)_k\,z^k}{(q,b_1,\dots,b_s;q)_k},
$$
where the symbol $(a_0,a_1,\dots,a_s;q)_k$ is a shortcut for $\prod_{\ell=0}^s(a_\ell;q)_k$.

\section{Proof of the principal theorems}
\label{sec2}

\subsection{Proof of Theorem \ref{main-1}}
\label{sec2.1}

We shall make use of the following $q$-congruences, which are special cases of \cite[Theorem 1.4]{GZ19a}.

\begin{lemma}
\label{lem:1}
Let $n$ be a positive integer coprime with $6$. Then
\begin{align*}
\sum_{k=0}^{(n-1)/2}[8k+1]\frac{(aq,q/a;q^2)_k (q;q^2)_{2k}}{(aq^6,q^6/a;q^6)_k (q^2;q^2)_{2k}}q^{2k^2}
&\equiv 0\pmod{[n]}, \\
\sum_{k=0}^{n-1}[8k+1]\frac{(aq,q/a;q^2)_k (q;q^2)_{2k}}{(aq^6,q^6/a;q^6)_k (q^2;q^2)_{2k}}q^{2k^2}
&\equiv 0\pmod{[n]}.
\end{align*}
\end{lemma}

We need the following $q$-series identity (see \cite[Lemma 3.1]{GZ19a}), which plays an important role in our proof of $r=1$ instances of \eqref{ram1a-r} and \eqref{ram1b-r}.

\begin{lemma}\label{lem:2}
Let $n$ be a positive odd integer. Then
\begin{equation}
\sum_{k=0}^{(n-1)/2}[8k+1]\frac{(q^{1-n},q^{1+n};q^2)_k (q;q^2)_{2k}}{(q^{6-n},q^{6+n};q^6)_k (q^2;q^2)_{2k}}q^{2k^2}
=q^{(1-n)/2}[n]\biggl(\frac{-3}{n}\biggr).
\label{eq:lem3.1}
\end{equation}
\end{lemma}

In order to prove Theorem \ref{main-1}, we need to establish the following parametric generalization.

\begin{theorem}
\label{main-1gen}
Let $n>1$ be an integer coprime with $6$ and let $r\ge 1$. Then, modulo
$$
[n^r]\prod_{j=0}^{(n^{r-1}-1)/d}(1-aq^{(2j+1)n})(a-q^{(2j+1)n}),
$$
we have
\begin{align}
&\sum_{k=0}^{(n^r-1)/d}[8k+1]\frac{(aq,q/a;q^2)_k (q;q^2)_{2k}}{(aq^6,q^6/a;q^6)_k (q^2;q^2)_{2k}}q^{2k^2} \notag\\
&\quad
\equiv q^{(1-n)/2}[n]\biggl(\frac{-3}{n}\biggr)
\sum_{k=0}^{(n^{r-1}-1)/d}[8k+1]_{q^n}
\frac{(aq^n,q^n/a;q^{2n})_k (q^n;q^{2n})_{2k}}{(aq^{6n},q^{6n}/a;q^{6n})_k (q^{2n};q^{2n})_{2k}}q^{2nk^2},  \label{eq:main-1-a}
\end{align}
where $d=1,2$.
\end{theorem}

\begin{proof}
By Lemma~\ref{lem:1} with $n$ replaced by $n^r$, we see that
the left-hand side of \eqref{eq:main-1-a} is congruent to $0$
modulo $[n^r]$.  On the other hand, replacing $n$ by $n^{r-1}$ and $q$ by $q^n$ in Lemma~\ref{lem:1}, we conclude that
the summation on the right-hand side of \eqref{eq:main-1-a} is congruent to $0$ modulo $[n^{r-1}]_{q^n}$.
Furthermore, since $n$ is odd, it is easily seen that the $q$-integer $[n]$ is relatively prime to $1+q^k$ for any positive integer $k$,
and so it is also relatively prime to the denominators of the sum on the right-hand side of \eqref{eq:main-1-a} because
$$
\frac{(q^n;q^{2n})_{2k}}{(q^{2n};q^{2n})_{2k}}={4k\brack 2k}_{q^n}\frac{1}{(-q^n;q^n)_{2k}^2},
$$
where ${2k\brack k}_{q^n}=(q^n;q^n)_{2k}/(q^n;q^n)_k^2$ denotes the central $q$-binomial coefficient.
This proves that the right-hand side of \eqref{eq:main-1-a} is congruent to $0$ modulo
$[n][n^{r-1}]_{q^n}=[n^r]$; hence the $q$-congruence \eqref{eq:main-1-a} is true modulo $[n^r]$.

To show it also holds modulo
\begin{align}
\prod_{j=0}^{(n^{r-1}-1)/d}(1-aq^{(2j+1)n})(a-q^{(2j+1)n}), \label{eq:prod}
\end{align}
we only need to prove that both sides of \eqref{eq:main-1-a} are identical when we take $a=q^{-(2j+1)n}$ or $a=q^{(2j+1)n}$
for any $j$ with $0\le j\le (n^{r-1}-1)/d$, that is,
\begin{align}
&\sum_{k=0}^{(n^r-1)/d}[8k+1]\frac{(q^{1-(2j+1)n},q^{1+(2j+1)n};q^2)_k (q;q^2)_{2k}}
{(q^{6-(2j+1)n},q^{6+(2j+1)n};q^6)_k (q^2;q^2)_{2k}}q^{2k^2} \notag\\
&\qquad
= q^{(1-n)/2}[n]\biggl(\frac{-3}{n}\biggr) \sum_{k=0}^{(n^{r-1}-1)/d}[8k+1]_{q^n}
\frac{(q^{-2jn},q^{(2j+2)n};q^{2n})_k (q^n;q^{2n})_{2k}}
{(q^{(5-2j)n},q^{(2j+7)n};q^{6n})_k (q^{2n};q^{2n})_{2k}}q^{2nk^2}.  \label{eq:main-a-n}
\end{align}
It is easy to see that $(n^r-1)/d\ge ((2j+1)n-1)/2$ for $0\le j\le (n^{r-1}-1)/d$,
and $(q^{1-(2j+1)n};q^2)_k=0$ for $k>((2j+1)n-1)/2$.
By Lemma \ref{lem:2} the left-hand side of \eqref{eq:main-a-n} is equal to
$$
q^{(1-(2j+1)n)/2}[(2j+1)n]\bigg(\frac{-3}{(2j+1)n}\bigg).
$$
Likewise, the right-hand side of \eqref{eq:main-a-n} is equal to
$$
q^{(1-n)/2}[n]\biggl(\frac{-3}{n}\biggr) \cdot q^{-jn}[2j+1]_{q^n}\biggl(\frac{-3}{2j+1}\biggr)
=q^{(1-(2j+1)n)/2}[(2j+1)n]\biggl(\frac{-3}{(2j+1)n}\biggr).
$$
This proves \eqref{eq:main-a-n}. Namely, the $q$-congruence \eqref{eq:main-1-a} holds modulo \eqref{eq:prod}.
Since $[n^r]$ and \eqref{eq:prod}  are relatively prime polynomials, the proof of \eqref{eq:main-1-a} is complete.
\end{proof}

\begin{proof}[Proof of Theorem {\rm\ref{main-1}}]
It is not hard to see that the limit of \eqref{eq:prod} as $a\to 1$ has the factor
\begin{align*}
\begin{cases}
\prod_{j=1}^r\Phi_{n^j}(q)^{2n^{r-j}} &\text{if $d=1$,}\\[5pt]
\prod_{j=1}^r\Phi_{n^j}(q)^{n^{r-j}+1} &\text{if $d=2$.}
\end{cases}
\end{align*}
Note that the denominator of the left-hand side of \eqref{eq:main-1-a} is a multiple of that of the right-hand side of \eqref{eq:main-1-a}.
Since $\gcd(n,6)=1$, the factor related to $a$ of the former is
$$
(aq^6;q^6)_{(n^r-1)/d} (q^6/a;q^6)_{(n^r-1)/d},
$$
whose limit as $a\to 1$ only has the factor
\begin{align*}
\begin{cases}
\prod_{j=1}^r\Phi_{n^j}(q)^{2n^{r-j}-2} &\text{if $d=1$,}\\[5pt]
\prod_{j=1}^r\Phi_{n^j}(q)^{n^{r-j}-1} &\text{if $d=2$,}
\end{cases}
\end{align*}
related to $\Phi_n(q),\Phi_{n^2}(q),\ldots,\Phi_{n^r}(q)$. Hence, letting $a\to 1$ in \eqref{eq:main-1-a} we conclude that
\eqref{q4a} is true modulo $\prod_{j=1}^r\Phi_{n^j}(q)^3$, where one product $\prod_{j=1}^r\Phi_{n^j}(q)$ comes from $[n^r]$.

Finally, by \cite[Theorem 1.1]{GZ19a}  we obtain
\begin{align*}
\sum_{k=0}^{(n-1)/d}[8k+1]\frac{(q;q^2)_k^2 (q;q^2)_{2k}}{(q^2;q^2)_{2k}(q^6;q^6)_k^2}q^{2k^2}
\equiv 0 \pmod{[n]} \quad\text{for $d=1,2$}.
\end{align*}
Replacing $n$ by $n^r$ in the above congruences, we deduce that the left-hand sides of \eqref{q4a} and \eqref{q4b}
are congruent to $0$ modulo $[n^r]$, while letting $q\mapsto q^n$ and $n\mapsto n^{r-1}$ in the above congruences,
we see that the right-hand sides of them are congruent to 0 modulo $[n][n^{r-1}]_{q^n}=[n^r]$ as well.
This means that the $q$-congruences \eqref{q4a} and \eqref{q4b} hold modulo $[n^r]$.
The proof then immediately follows from the fact that the least common multiple of $\prod_{j=1}^r\Phi_{n^j}(q)^3$ and $[n^r]$
is just $[n^r]\prod_{j=1}^r\Phi_{n^j}(q)^2$.
\end{proof}

\subsection{Proof of Theorem \ref{main-2}}
\label{sec2.2}
Similarly to what we did above, we need the following $q$-congruence and $q$-identity; they follow from the $b\to 0$ case of \cite[Theorem 4.8]{GZ19a}.

\begin{lemma}\label{lem:3}
Let $n$ be a positive odd integer, and $d=1$ or $2$. Then
\begin{align}
\sum_{k=0}^{(n-1)/d}[3k+1]\frac{(aq,q/a;q^2)_k (q;q^2)_k}{(aq,q/a;q)_k (q^2;q^2)_k} q^{-{k+1\choose 2} }
&\equiv 0\pmod{[n]}, \label{lem-3-1} \\
\sum_{k=0}^{(n-1)/2}[3k+1]\frac{(q^{1-n},q^{1+n};q^2)_k (q;q^2)_{k}}{(q^{1-n},q^{1+n};q)_k (q^2;q^2)_{k}}q^{-{k+1\choose 2}}
&=q^{(1-n)/2}[n].  \label{lem-3-2}
\end{align}
\end{lemma}

For a real number $x$, we use the standard notation $\lfloor x\rfloor$ and $\lceil x\rceil$ for the floor (integer part) and ceiling functions;
these integers satisfy $\lfloor x\rfloor\le x\le\lceil x\rceil$.
We have the following parametric generalization of Theorem~\ref{main-2}.

\begin{theorem}\label{main-2-par}
Let $n>1$ be an odd integer and $r\ge 1$. Then, modulo
$$
[n^r]\prod_{j=\lceil (n^{r-1}-1)/(2d)\rceil}^{(n^{r-1}-1)/d}(1-aq^{(2j+1)n})(a-q^{(2j+1)n}),
$$
we have
\begin{align}
&\sum_{k=0}^{(n^r-1)/d}[3k+1]\frac{(aq,q/a;q^2)_k (q;q^2)_k}{(aq,q/a;q)_k (q^2;q^2)_k} q^{-{k+1\choose 2}} \notag\\
&\qquad
\equiv q^{(1-n)/2}[n]
\sum_{k=0}^{(n^{r-1}-1)/d}[3k+1]_{q^n}
\frac{(aq^n,q^n/a;q^{2n})_k (q^n;q^{2n})_{k}}{(aq^n,q^n/a;q^n)_k (q^{2n};q^{2n})_{k}}q^{-nk{k+1\choose 2}},  \label{eq:main-2-a}
\end{align}
where $d=1,2$.
\end{theorem}

\begin{proof}
Replacing $n$ by $n^r$ in \eqref{lem-3-1}, we see that
the left-hand side of \eqref{eq:main-2-a} is congruent to $0$
modulo $[n^r]$.  Moreover, replacing $n$ by $n^{r-1}$ and $q$ by $q^n$ in \eqref{lem-3-1} means that
the right-hand side of \eqref{eq:main-2-a} is congruent to $0$ modulo $[n][n^{r-1}]_{q^n}=[n^r]$. That is, the $q$-congruence \eqref{eq:main-2-a} holds modulo $[n^r]$.

To prove it is also true modulo
\begin{align}
\prod_{j=\lceil (n^{r-1}-1)/(2d)\rceil}^{(n^{r-1}-1)/d}(1-aq^{(2j+1)n})(a-q^{(2j+1)n}),\label{eq:prod-2}
\end{align}
it suffices to show that both sides of \eqref{eq:main-2-a} are equal for all $a=q^{-(2j+1)n}$ and $a=q^{(2j+1)n}$
with $(n^{r-1}-1)/(2d)\le j\le (n^{r-1}-1)/d$, i.e.,
\begin{align}
&\sum_{k=0}^{(n^r-1)/d}[3k+1]\frac{(q^{1-(2j+1)n},q^{1+(2j+1)n};q^2)_k (q;q^2)_{k}}{(q^{1-(2j+1)n},q^{1+(2j+1)n};q)_k (q^2;q^2)_{k}}q^{-{k+1\choose 2}} \notag\\
&\qquad
= q^{(1-n)/2} [n] \sum_{k=0}^{(n^{r-1}-1)/d}[3k+1]_{q^n}
\frac{(q^{-2jn},q^{(2j+2)n};q^{2n})_k (q^n;q^{2n})_{k}}{(q^{-2jn},q^{(2j+2)n};q^n)_k (q^{2n};q^{2n})_{k}}q^{-n{k+1\choose 2}}.  \label{eq:main-b-n}
\end{align}
It is easy to see that $(n^r-1)/d \ge ((2j+1)n-1)/2 $ and
$(2j+1)n >  (n^r-1)/d $
for $\lceil( n^{r-1}-1)/2d\rceil\le j\le (n^{r-1}-1)/d$.
Hence, the left-hand side of \eqref{eq:main-b-n} is well-defined (the denominator is non-zero) and is equal to
\begin{align*}
&\sum_{k=0}^{((2j+1)n-1)/2}[3k+1]\frac{(q^{1-(2j+1)n},q^{1+(2j+1)n};q^2)_k (q;q^2)_{k}}{(q^{1-(2j+1)n},q^{1+(2j+1)n};q)_k (q^2;q^2)_{k}}q^{-{k+1\choose 2}} \\
&\qquad
=q^{(1-(2j+1)n)/2}[(2j+1)n]
\end{align*}
by \eqref{lem-3-2}. Similarly, the right-hand side of \eqref{eq:main-b-n} is equal to
$$
q^{(1-n)/2}[n] \cdot q^{-jn}[2j+1]_{q^n}=q^{(1-(2j+1)n)/2}[(2j+1)n],
$$
and so the identity \eqref{eq:main-b-n} holds. Namely, the $q$-congruence \eqref{eq:main-2-a} is true modulo \eqref{eq:prod-2}.
This completes the proof of \eqref{eq:main-2-a}.
\end{proof}

\begin{proof}[Proof of Theorem {\rm\ref{main-2}}]
This time the limit of \eqref{eq:prod-2} as $a\to 1$ has the factor
\begin{align*}
\begin{cases}
\prod_{j=1}^r\Phi_{n^j}(q)^{n^{r-j}+1} &\text{if $d=1$,}\\[5pt]
\prod_{j=1}^r\Phi_{n^j}(q)^{n^{r-j}+1-2\lfloor(n^{r-j}+1)/4\rfloor} &\text{if $d=2$,}
\end{cases}
\end{align*}
where in the $d=2$ case we use the fact the set $\{(2j+1)n:j=0,\ldots, \lfloor(n^{r-1}-3)/4\rfloor\}$
contains exactly $\lfloor(n^{r-j}+1)/4\rfloor$ multiples of $n^j$ for $j=1,\dots,r$.

On the other hand, the denominator of (the reduced form of) the left-hand side of \eqref{eq:main-2-a} is a multiple of that of the right-hand side of \eqref{eq:main-2-a}.
The factor related to $a$ of the denominator is
$$
\begin{cases}
\dfrac{(aq,q/a;q)_{n^r-1}}{(aq,q/a;q^2)_{(n^r-1)/2}}
=(aq^2,q^2/a;q^2)_{(n^r-1)/2} &\text{if $d=1$,}\\[10pt]
\dfrac{(aq,q/a;q)_{(n^r-1)/2}}{(aq,q/a;q^2)_{\lceil(n^r-1)/4\rceil}}
=(aq^2,q^2/a;q^2)_{\lfloor(n^r-1)/4\rfloor} &\text{if $d=2$}.
\end{cases}
$$
Its limit as $a\to 1$ only has the following factor
\begin{align*}
\begin{cases}
\prod_{j=1}^r\Phi_{n^j}(q)^{n^{r-j}-1} &\text{if $d=1$,}\\[5pt]
\prod_{j=1}^r\Phi_{n^j}(q)^{2\lfloor(n^{r-j}-1)/4\rfloor} &\text{if $d=2$,}
\end{cases}
\end{align*}
related to $\Phi_n(q),\Phi_{n^2}(q),\ldots,\Phi_{n^r}(q)$.
Therefore, setting $a\to 1$ in \eqref{eq:main-2-a}, we conclude that
\eqref{eq:q-div-WZ-2} holds modulo $\prod_{j=1}^r\Phi_{n^j}(q)^3$, where one product $\prod_{j=1}^r\Phi_{n^j}(q)$ is from $[n^r]$.

Finally, along the lines of proof of Theorem \ref{main-1}, using the following $q$-congruences from \cite{Gu20a}:
\begin{align*}
\sum_{k=0}^{(n-1)/d}[3k+1]\frac{(q;q^2)_k^3 q^{-{k+1\choose 2} } }{(q;q)_k^2 (q^2;q^2)_k}
\equiv 0 \pmod{[n]} \quad\text{for $d=1,2$},
\end{align*}
we can prove that the $q$-congruences \eqref{eq:q-div-WZ-1} and \eqref{eq:q-div-WZ-2} hold modulo $[n^r]$, thus
completing the proof of the theorem.
\end{proof}

\section{More Dwork-type $q$-congruences}
\label{sec3}

Throughout this section, $p$ always denotes an odd prime.
Below we give $q$-analogues of some known or conjectural Dwork-type congruences.
In particular, we completely confirm the supercongruence conjectures (B.3), (L.3) of Swisher \cite{Sw15}
and also confirm the first cases of her conjectures (E.3) and (F.3).

\subsection{Another $q$-analogue of \eqref{3k+1-a} and \eqref{3k+1-b}}
From \cite{Gu19c,GZ20} we see that supercongruences may have different $q$-analogues.
Here we show that the supercongruences \eqref{3k+1-a} and \eqref{3k+1-b} fall into this category
and possess $q$-analogues different from those presented in Theorem~\ref{main-2}.

\begin{theorem}\label{main-new}
Let $n>1$ be odd and let $r\ge 1$. Then, modulo $[n^r]\prod_{j=1}^r\Phi_{n^j}(q)^2$,
\begin{align}
\sum_{k=0}^{(n^r-1)/d}[3k+1]\frac{(q;q^2)_k^3 (-1;q)_k q^k }{(q;q)_k^3 (-q^2;q)_{2k}}
\equiv \frac{1+q}{1+q^n}[n]\sum_{k=0}^{(n^{r-1}-1)/2}[3k+1]_q\frac{(q^n;q^{2n})_k^3(-1;q^n)_k q^{nk} }{(q^n;q^n)_k^3 (-q^{2n};q^{n})_{2k}}, \label{eq:q-div-new-1}
\end{align}
where $d=1,2$.
\end{theorem}

\begin{proof}[Sketch of proof]
Letting $b=-1$ in \cite[Theorem 4.8]{GZ19a}, we get the following $q$-cong\-ruence:
modulo $[n](1-aq^n)(a-q^n)$,
\begin{align}
\sum_{k=0}^{(n-1)/d}[3k+1]\frac{(aq,q/a,q;q^2)_k  (-1;q)_k}{(aq,q/a,q;q)_k (-q^2;q)_{2k}} q^k
\equiv \frac{1+q}{1+q^n}[n],   \label{eq:div-3-new}
\end{align}
where $d=1,2$. This means that the left-hand side of \eqref{eq:div-3-new} is congruent to $0$ modulo $[n]$, and also (when $a=q^n$) that
$$
\sum_{k=0}^{(n-1)/2}
[3k+1]\frac{(q^{1-n},q^{1+n},q;q^2)_k  (-1;q)_k}{(q^{1-n},q^{1+n},q;q)_k (-q^2;q)_{2k}} q^k
=\frac{1+q}{1+q^n}[n].
$$
Thus, like in the proof of Theorem \ref{main-2}, we can establish the following parametric generalization of \eqref{eq:q-div-new-1}:
modulo
$$
[n^r]\prod_{j=\lceil (n^{r-1}-1)/(2d)\rceil}^{(n^{r-1}-1)/d}(1-aq^{(2j+1)n})(a-q^{(2j+1)n}),
$$
we have
\begin{align}
&\sum_{k=0}^{(n^r-1)/d}[3k+1]\frac{(aq,q/a,q;q^2)_k  (-1;q)_k}{(aq,q/a,q;q)_k (-q^2;q)_{2k}} q^k \notag\\
&\qquad
\equiv \frac{1+q}{1+q^n}[n]\sum_{k=0}^{(n^{r-1}-1)/d}[3k+1]_{q^n}
 \frac{(aq^{n},q^{n}/a,q^{n};q^{2n})_k (-1;q^n)_k}
{(aq^{n},q^{n}/a,q^{n};q^{n})_k (-q^{2n};q^{n})_{2k}} q^{nk}, \label{main-new-a}
\end{align}
where $d=1,2$.

Letting $a\to 1$ in \eqref{main-new-a}, we conclude that the $q$-congruence \eqref{eq:q-div-new-1} is true modulo $\prod_{j=1}^r\Phi_{n^j}(q)^3$.
Note that the proof of \cite[Theorem 6.1]{GS20b} also implies that \eqref{eq:div-3-new} modulo $[n]$ holds for $a=1$.
Applying this $q$-congruence on both sides of \eqref{eq:q-div-new-1}, we deduce that \eqref{eq:q-div-new-1}
are also true modulo $[n^r]$.
\end{proof}

\subsection{Another `divergent' Dwork-type supercongruence}
Guillera and the second author \cite{GZ12} proved the following
`divergent' Ramanujan-type supercongruence:
\begin{align}
\sum_{k=0}^{(p-1)/2} \frac{(\frac{1}{2})_k^3}{k!^3}(3k+1)(-1)^k 2^{3k} &\equiv p\bigg(\frac{-1}{p}\bigg)\pmod{p^3} \label{eq:div-3}
\end{align}
(see also \cite[Conjecture 5.1(ii)]{Su11}).
The first author \cite{Gu20a} gave a $q$-analogue of \eqref{eq:div-3} and proposed the following conjecture on Dwork-type supercongruences:
\begin{align}
\sum_{k=0}^{(p^r-1)/2} \frac{(\frac{1}{2})_k^3}{k!^3}(3k+1)(-1)^k 2^{3k}
&\equiv p\bigg(\frac{-1}{p}\bigg) \sum_{k=0}^{(p^{r-1}-1)/2} \frac{(\frac{1}{2})_k^3}{k!^3}(3k+1)(-1)^k 2^{3k} \pmod{p^{3r+\delta_{p,3}}}, \label{eq:last-1} \\
\sum_{k=0}^{p^r-1} \frac{(\frac{1}{2})_k^3}{k!^3}(3k+1)(-1)^k 2^{3k}
&\equiv p\bigg(\frac{-1}{p}\bigg) \sum_{k=0}^{p^{r-1}-1} \frac{(\frac{1}{2})_k^3}{k!^3}(3k+1)(-1)^k 2^{3k} \pmod{p^{3r}}.  \label{eq:last-2}
\end{align}

In the spirit of Theorems  \ref{main-1} and \ref{main-2}, we have the following $q$-generalization of the above two supercongruences modulo $p^{3r}$.

\begin{theorem}\label{main-3}
Let $n>1$ be odd and let $r\ge 1$. Then, modulo $[n^r]\prod_{j=1}^r\Phi_{n^j}(q)^2$,
\begin{align}
&\sum_{k=0}^{(n^r-1)/d}(-1)^k[3k+1]\frac{(q;q^2)_k^3 (-q;q)_k}{(q;q)_k^3 (-q^2;q^2)_k} q^{-{k+1\choose 2}}  \notag\\
&\qquad \equiv q^{(1-n)/2}[n]\bigg(\frac{-1}{n}\bigg)
\sum_{k=0}^{(n^{r-1}-1)/d}(-1)^k[3k+1]_{q^n}\frac{(q^{n};q^{2n})_k^3 (-q^n;q^n)_k}{(q^{n};q^{n})_k^3 (-q^{2n};q^{2n})_k} q^{-n{k+1\choose 2}}, \label{eq:div-3-2}
\end{align}
where $d=1,2$.
\end{theorem}

\begin{proof}[Sketch of proof]
Letting $b=-1$ and $c\to 0$ in \cite[Theorem 6.1]{GS20b} (see also \cite[Conjecture 4.6]{GZ19a}), we get the following $q$-congruence:
modulo $[n](1-aq^n)(a-q^n)$,
\begin{align}
\sum_{k=0}^{(n-1)/d}(-1)^k[3k+1]\frac{(aq,q/a,q;q^2)_k  (-q;q)_k}{(aq,q/a,q;q)_k (-q^2;q^2)_k} q^{-{k+1\choose 2}}
\equiv q^{(1-n)/2}[n]\bigg(\frac{-1}{n}\bigg),   \label{eq:div-3-3}
\end{align}
where $d=1,2$. Namely, the left-hand side of \eqref{eq:div-3-2} is congruent to $0$ modulo $[n]$, and
$$
\sum_{k=0}^{(n-1)/2}(-1)^k[3k+1]\frac{(q^{1-n},q^{1+n},q;q^2)_k  (-q;q)_k}
{(q^{1-n},q^{1+n},q;q)_k (-q^2;q^2)_k} q^{-{k+1\choose 2}}
=q^{(1-n)/2}[n]\bigg(\frac{-1}{n}\bigg).
$$
Thus, we may establish a parametric generalization of \eqref{eq:div-3-2}:
modulo
$$
[n^r]\prod_{j=\lceil (n^{r-1}-1)/(2d)\rceil}^{(n^{r-1}-1)/d}(1-aq^{(2j+1)n})(a-q^{(2j+1)n}),
$$
we have
\begin{align}
&\sum_{k=0}^{(n^r-1)/d}(-1)^k[3k+1]\frac{(aq,q/a,q;q^2)_k (-q;q)_k}{(aq,q/a,q;q)_k (-q^2;q^2)_k} q^{-{k+1\choose 2}}  \notag\\
&\qquad \equiv q^{(1-n)/2}[n]\bigg(\frac{-1}{n}\bigg)\sum_{k=0}^{(n^{r-1}-1)/d}(-1)^k[3k+1]_{q^n} \notag\\
&\qquad\quad\times \frac{(aq^{n},q^{n}/a,q^{n};q^{2n})_k (-q^n;q^n)_k}
{(aq^{n},q^{n}/a,q^{n};q^{n})_k (-q^{2n};q^{2n})_k} q^{-n{k+1\choose 2}}, \label{main-3-a}
\end{align}
where $d=1,2$.

Letting $a\to 1$ in \eqref{main-3-a}, we know that \eqref{eq:div-3-2} holds modulo $\prod_{j=1}^r\Phi_{n^j}(q)^3$.
Applying the $q$-congruence \eqref{eq:div-3-3} modulo $[n]$ with $a=1$ on both sides of \eqref{eq:div-3-2}, we conclude that \eqref{eq:div-3-2}
is also true modulo $[n^r]$.
\end{proof}

\subsection{Two supercongruences of Swisher}
Swisher's conjectural supercongruence (B.3) from \cite{Sw15} can be stated as follows:
\begin{align*}
&\sum_{k=0}^{(p^r-1)/2}(-1)^k (4k+1)\frac{(\frac{1}{2})_k^3}{k!^3} \\[5pt]
&\qquad \equiv
\begin{cases}
\displaystyle p\bigg(\frac{-1}{p}\bigg) \sum_{k=0}^{(p^{r-1}-1)/2} (-1)^k(4k+1)\frac{(\frac{1}{2})_k^3}{k!^3} \pmod{p^{3r}} &\text{if $p\equiv 1\pmod4$},\\[15pt]
\displaystyle p^2 \sum_{k=0}^{(p^{r-2}-1)/2}
(-1)^k(4k+1)\frac{(\frac{1}{2})_k^3}{k!^3} \pmod{p^{3r-2}}
&\text{if $p\equiv 3\pmod4$, $r\ge 2$}.
\end{cases}
\end{align*}
In fact we find out that, more generally, for any prime $p>2$,
\begin{align}
\sum_{k=0}^{(p^r-1)/2}(-1)^k (4k+1)\frac{(\frac{1}{2})_k^3}{k!^3}
\equiv p\bigg(\frac{-1}{p}\bigg) \sum_{k=0}^{(p^{r-1}-1)/2}
(-1)^k(4k+1)\frac{(\frac{1}{2})_k^3}{k!^3} \pmod{p^{3r}}.
\label{eq:b3-new-1}
\end{align}
Observe that Swisher's supercongruence (B.3) for $p\equiv 3\pmod{4}$ follows from using \eqref{eq:b3-new-1} twice.
It is natural to conjecture that the following companion supercongruence of \eqref{eq:b3-new-1} is also true:
\begin{align}
\sum_{k=0}^{p^r-1}(-1)^k (4k+1)\frac{(\frac{1}{2})_k^3}{k!^3} \equiv
p\bigg(\frac{-1}{p}\bigg) \sum_{k=0}^{p^{r-1}-1}
(-1)^k(4k+1)\frac{(\frac{1}{2})_k^3}{k!^3} \pmod{p^{3r}}.
\label{eq:b3-new-2}
\end{align}
Here we prove the Dwork-type supercongruences \eqref{eq:b3-new-1} and \eqref{eq:b3-new-2}
by establishing the following $q$-analogues.

\begin{theorem}\label{main-4}
Let $n>1$ be odd and let $r\ge 1$. Then, modulo $[n^r]\prod_{j=1}^r\Phi_{n^j}(q)^2$,
\begin{align}
&\sum_{k=0}^{(n^r-1)/d}(-1)^k[4k+1]\frac{(q;q^2)_k^2 (q^2;q^4)_k}{(q^2;q^2)_k^2 (q^4;q^4)_k}  \notag\\
&\qquad \equiv q^{(1-n)/2}[n]\bigg(\frac{-1}{n}\bigg)
\sum_{k=0}^{(n^{r-1}-1)/d}(-1)^k[4k+1]_{q^n}\frac{(q^n;q^{2n})_k^2 (q^{2n};q^{4n})_k}{(q^{2n};q^{2n})_k^2 (q^{4n};q^{4n})_k}, \label{q-b3}
\end{align}
where $d=1,2$.
\end{theorem}

\begin{proof}[Sketch of proof]
Letting $c=-1$ in \cite[Theorem 4.2]{GZ19a}, we obtain the following $q$-cong\-ruence for
odd $n$: modulo $[n](1-aq^n)(a-q^n)$,
\begin{align}
\sum_{k=0}^{(n-1)/d}(-1)^k[4k+1]\frac{(aq,q/a;q^2)_k(q^2;q^4)_k} {(aq^2,q^2/a;q^2)_k (q^4;q^4)_k}
\equiv q^{(1-n)/2}[n]\bigg(\frac{-1}{n}\bigg),  \label{qb2-new}
\end{align}
where $d=1,2$. That is to say, the left-hand side of  \eqref{qb2-new} is congruent to $0$ modulo $[n]$, and
\begin{align*}
\sum_{k=0}^{(n-1)/2}(-1)^k[4k+1]\frac{(q^{1-n},q^{1+n};q^2)_k (q^2;q^4)_k} {(q^{2-n},q^{2+n};q^2)_k (q^4;q^4)_k}=q^{(1-n)/2}[n]\bigg(\frac{-1}{n}\bigg).
\end{align*}
Along the lines of our proof of Theorem \ref{main-1}, we can prove the following parametric version of \eqref{q-b3}:
modulo
$$
[n^r]\prod_{j=0}^{(n^{r-1}-1)/d}(1-aq^{(2j+1)n})(a-q^{(2j+1)n}),
$$
we have
\begin{align}
&\sum_{k=0}^{(n^r-1)/d}(-1)^k[4k+1]\frac{(aq,q/a;q^2)_k (q^2;q^4)_k}{(aq^2,q^2/a;q^2)_k (q^4;q^4)_k}  \notag\\
&\qquad \equiv q^{(1-n)/2}[n]\bigg(\frac{-1}{n}\bigg)
\sum_{k=0}^{(n^{r-1}-1)/d}(-1)^k[4k+1]_{q^n}\frac{(aq^n,q^n/a;q^{2n})_k (q^{2n};q^{4n})_k}{(aq^{2n},q^{2n}/a;q^{2n})_k (q^{4n};q^{4n})_k},  \label{eq:main-4-a}
\end{align}
where $d=1,2$.

Letting $a\to 1$ in \eqref{eq:main-4-a}, we see that \eqref{q-b3} is true modulo $\prod_{j=1}^r\Phi_{n^j}(q)^3$.
Note that the proof of \cite[Theorem 4.2]{GZ19a} also indicates that the $q$-congruence \eqref{qb2-new} modulo $[n]$ hold for $a=1$.
Applying this $q$-congruence on both sides of \eqref{q-b3}, we conclude that \eqref{q-b3}
is also true modulo $[n^r]$.
\end{proof}

Swisher \cite[Conjecture (L.3)]{Sw15}
conjectured that, for $r\ge 1$,
\begin{align}
\sum_{k=0}^{(p^r-1)/2} (-1)^k(6k+1)\frac{(\frac{1}{2})_k^3}{k!^3 8^k}
\equiv p\bigg(\frac{-2}{p}\bigg) \sum_{k=0}^{(p^{r-1}-1)/2} (-1)^k(6k+1)\frac{(\frac{1}{2})_k^3}{k!^3 8^k} \pmod{p^{3r}}. \label{L3-1}
\end{align}
Recently, the first author \cite[Conjecture 4.5]{Gu18a} made the following similar conjecture:
\begin{align}
\sum_{k=0}^{p^r-1} (-1)^k(6k+1)\frac{(\frac{1}{2})_k^3}{k!^3 8^k}
\equiv p\bigg(\frac{-2}{p}\bigg) \sum_{k=0}^{p^{r-1}-1} (-1)^k(6k+1)\frac{(\frac{1}{2})_k^3}{k!^3 8^k} \pmod{p^{3r}}.   \label{L3-2}
\end{align}
We confirm the supercongruences \eqref{L3-1} and \eqref{L3-2} by establishing the following Dwork-type $q$-supercongruence.

\begin{theorem}\label{main-5}
Let $n>1$ be odd and let $r\ge 1$. Then, modulo $[n^r]\prod_{j=1}^r\Phi_{n^j}(q)^2$,
\begin{align}
&\sum_{k=0}^{(n^r-1)/d}(-1)^k[6k+1]\frac{(q;q^2)_k^3 (-q^2;q^4)_k}{(q^4;q^4)_k^3 (-q;q^2)_k} q^{k^2}  \notag\\
&\qquad \equiv q^{(1-n)/2}[n]\bigg(\frac{-2}{n}\bigg)
\sum_{k=0}^{(n^{r-1}-1)/d}(-1)^k[6k+1]_{q^n}\frac{(q^n;q^{2n})_k^3 (-q^{2n};q^{4n})_k}{(q^{4n};q^{4n})_k^3 (-q^n;q^{2n})_k} q^{nk^2}, \label{eq:div-2-1}
\end{align}
where $d=1,2$.
\end{theorem}

\begin{proof}[Sketch of proof]
Setting $b=-q^2$ in \cite[Theorem 4.5]{GZ19a}, we are led to the following $q$-congruence:
modulo $[n](1-aq^n)(a-q^n)$,
\begin{align}
\sum_{k=0}^{(n^r-1)/d}(-1)^k[6k+1]\frac{(aq,q/a,q;q^2)_k (-q^2;q^4)_k}{(aq^4,q^4/a,q^4;q^4)_k  (-q;q^2)_k} q^{k^2}
\equiv q^{(1-n)/2}[n]\bigg(\frac{-2}{n}\bigg). \label{eq:div-2-2a}
\end{align}
Thus, we can prove the following parametric version of \eqref{q-b3}: modulo
$$
[n^r]\prod_{j=0}^{(n^{r-1}-1)/d}(1-aq^{(2j+1)n})(a-q^{(2j+1)n}),
$$
we have
\begin{align}
&\sum_{k=0}^{(n^r-1)/d}(-1)^k[6k+1]\frac{(aq,q/a,q;q^2)_k  (-q^2;q^4)_k}{(aq^4,q^4/a,q^4;q^4)_k (-q;q^2)_k} q^{k^2}  \notag\\
&\quad \equiv q^{(1-n)/2}[n]\bigg(\frac{-2}{n}\bigg)
\sum_{k=0}^{(n^{r-1}-1)/d}(-1)^k[6k+1]_{q^n}
\frac{(aq^n,q^n/a,q^n;q^{2n})_k (-q^{2n};q^{4n})_k}{(aq^{4n},q^{4n}/a,q^{4n};q^{4n})_k (-q^n;q^{2n})_k} q^{nk^2},  \label{eq:main-5-a}
\end{align}
where $d=1,2$. The proof of  \eqref{eq:div-2-1} modulo $\prod_{j=1}^r\Phi_{n^j}(q)^3$ then follows by taking the limit as $a\to 1$ in \eqref{eq:main-5-a}, and the proof of
\eqref{eq:div-2-1} modulo $[n^r]$ follows from the $q$-congruence \eqref{eq:div-2-2a} modulo $[n]$ with $a=1$.
\end{proof}

\subsection{Another two supercongruences from Swisher's list}
In \cite[Conjectures (E.3), (F.3)]{Sw15} Swisher proposed the following conjectures:
\begin{align}
&\sum_{k=0}^{(p^r-1)/3} \frac{(6k+1)(\frac{1}{3})_k^3}{k!^3 (-1)^k}
\equiv p  \sum_{k=0}^{(p^{r-1}-1)/3} \frac{(6k+1)(\frac{1}{3})_k^3}{k!^3(-1)^k} \pmod{p^{3r}}\quad\text{for $p\equiv 1\pmod 3$}, \label{e3}\\
&\sum_{k=0}^{(p^r-1)/4}\frac{(8k+1)(\frac{1}{4})_k^3}{k!^3 (-1)^k }
\equiv p\left(\frac{-2}{p}\right)  \sum_{k=0}^{(p^{r-1}-1)/4} \frac{(8k+1)(\frac{1}{4})_k^3}{k!^3 (-1)^k} \pmod{p^{3r}}\quad\text{for $p\equiv 1\pmod 4$}.  \label{f3}
\end{align}
Here we confirm \eqref{e3} and \eqref{f3} by showing the following $q$-analogues.

\begin{theorem}\label{main-e}
Let $n>1$ be an integer with $n\equiv 1\pmod{6}$ and let $r\ge 1$. Then, modulo $[n^r]_{q^2}\prod_{j=1}^r\Phi_{n^j}(q^2)^2$,
\begin{align}
&\sum_{k=0}^{(n^r-1)/d}(-1)^k[6k+1]_{q^2}\frac{(q^2;q^6)_k^3(-q^3;q^6)_k}{(q^6;q^6)_k^3 (-q^5;q^6)_k} q^k  \notag\\
&\qquad \equiv q^{1-n}[n]_{q^2}
\sum_{k=0}^{(n^{r-1}-1)/d}(-1)^k[6k+1]_{q^{2n}}\frac{(q^{2n};q^{6n})_k^3(-q^{3n};q^{6n})_k}{(q^{6n};q^{6n})_k^3 (-q^{5n};q^{6n})_k} q^{nk}, \label{q-e3}
\end{align}
where $d=1,3$.
\end{theorem}

\begin{proof}[Sketch of proof] It is easy to see that \cite[Theorem 4.2]{GZ19a} can be generalized as follows.
Modulo $[n](1-aq^n)(a-q^n)$,
\begin{align}
&\sum_{k=0}^{(n-1)/d}[2mk+1]\frac{(aq,q/a,q/c,q;q^m)_k}
{(aq^m,q^m/a,cq^m,q^m;q^m)_k} c^kq^{(m-2)k}  \notag\\
&\qquad\equiv\frac{(c/q)^{(n-1)/m} (q^2/c;q^m)_{(n-1)/m}}{(cq^m;q^m)_{(n-1)/m}}[n] \quad\text{for $n\equiv 1\pmod{m}$},
\label{eq:q-Long-gen}
\end{align}
where $d=1$ or $m$.
Here we emphasize that, in order to prove \eqref{eq:q-Long-gen} holds modulo $[n]$, we need to show that
\begin{align*}
\sum_{k=0}^{n-1}[2mk+1]\frac{(aq,q/a,q/c,q;q^m)_k}
{(aq^m,q^m/a,cq^m,q^m;q^m)_k} c^kq^{(m-2)k}
\equiv 0\pmod{\Phi_n(q)}
\end{align*}
is true for all integers $n>1$ with $\gcd(n,m)=1$.
Then we use the same arguments as \cite[Theorems 1.2 and 1.3]{GZ19a} to deal with the modulus $[n]$ case.

We now put $m=3$, $q\mapsto q^2$ and $c=-q^{-1}$ in \eqref{eq:q-Long-gen} to get
\begin{align*}
&\sum_{k=0}^{(n-1)/d}(-1)^k[6k+1]_{q^2}\frac{(aq^2,q^2/a,q^2,-q^3;q^6)_k}
{(aq^6,q^6/a,q^6,-q^5;q^6)_k} q^k \notag\\
&\qquad\equiv q^{1-n}[n]_{q^2}(-1)^{n-1} \pmod{\Phi_n(q^2)(1-aq^{2n})(a-q^{2n})} \quad\text{for $n\equiv 1\pmod{6}$},
\end{align*}
where $d=1,3$. Using this $q$-congruence, we can produce a generalization of \eqref{q-e3} with an extra parameter~$a$:
modulo
$$
[n^r]_{q^2}\prod_{j=0}^{(n^{r-1}-1)/d}(1-aq^{(6j+2)n})(a-q^{(6j+2)n}),
$$
we have
{\small\begin{align*}
&\sum_{k=0}^{(n^r-1)/d}(-1)^k[6k+1]_{q^2}\frac{(aq^2,q^2/a,q^2,-q^3;q^6)_k}
{(aq^6,q^6/a,q^6,-q^5;q^6)_k} q^k  \notag\\
&\qquad \equiv q^{1-n}[n]_{q^2}
\sum_{k=0}^{(n^{r-1}-1)/d}(-1)^k[6k+1]_{q^{2n}}
\frac{(aq^{2n},q^{2n}/a,q^{2n},-q^{3n};q^{6n})_k}
{(aq^{6n},q^{6n}/a,q^{6n},-q^{5n};q^{6n})_k} q^{nk},
\end{align*}}%
where $d=1,3$.
\end{proof}

It is easy to see that, when $n=p$ and $q\to 1$, the $q$-supercongruence \eqref{q-e3} for $d=3$ reduces to \eqref{e3}, and
it for $d=1$  confirms the first supercongruence in \cite[Conjecture 5.3]{Gu19a}. Moreover, letting $n=p$ and $q\to -1$ in \eqref{q-e3},
we obtain the following new Dwork-type supercongruence: for $p\equiv 1\pmod{3}$,
\begin{align*}
&\sum_{k=0}^{(p^r-1)/d} (6k+1) \frac{(\frac{1}{3})_k^3 (\frac{1}{2})_k}{k!^3 (\frac{5}{6})_k}
\equiv p  \sum_{k=0}^{(p^{r-1}-1)/d} (6k+1) \frac{(\frac{1}{3})_k^3 (\frac{1}{2})_k}{k!^3 (\frac{5}{6})_k} \pmod{p^{3r}},
\end{align*}
where $d=1,3$.

When $r$ is even and $p>3$, we always have $p^2\equiv 1\pmod{24}$. Thus, letting $n=p^2$, $r\mapsto r/2$ and $q\to 1$ in \eqref{q-e3}
we arrive at
\begin{align}
\sum_{k=0}^{(p^r-1)/3} \frac{(6k+1)(\frac{1}{3})_k^3}{k!^3 (-1)^k}
\equiv p^2  \sum_{k=0}^{(p^{r-2}-1)/3} \frac{(6k+1)(\frac{1}{3})_k^3}{k!^3(-1)^k} \pmod{p^{2r}}\quad \text{for $r\ge 2$ even}.  \label{e-even}
\end{align}
This partially confirm the second case of \cite[Conjecture (E.3)]{Sw15}, which asserts that \eqref{e-even} holds modulo $p^{3r-2}$ for $p\equiv 2\pmod{3}$.

\begin{theorem}\label{main-f}
Let $n>1$ be an integer with $n\equiv 1\pmod{4}$ and let $r\ge 1$. Then, modulo $[n^r]\prod_{j=1}^r\Phi_{n^j}(q)^2$,
\begin{align}
&\sum_{k=0}^{(n^r-1)/d}(-1)^k[8k+1]\frac{(q;q^4)_k^3 (-q^2;q^4)_k}{(q^4;q^4)_k^3 (-q^3;q^4)_k} q^k  \notag\\
&\quad \equiv q^{(1-n)/2}[n]\bigg(\frac{-2}{n}\bigg)
\sum_{k=0}^{(n^{r-1}-1)/d}(-1)^k[8k+1]_{q^n}\frac{(q^n;q^{4n})_k^3 (-q^{2n};q^{4n})_k}{(q^{4n};q^{4n})_k^3 (-q^{3n};q^{2n})_k} q^{nk}, \label{q-f3}
\end{align}
where $d=1,4$.
\end{theorem}

\begin{proof}[Sketch of proof]
This time we take $m=4$ and $c=-q^{-1}$ in \eqref{eq:q-Long-gen} to get
\begin{align*}
&\sum_{k=0}^{(n-1)/d}(-1)^k[8k+1]\frac{(aq,q/a,q,-q^2;q^4)_k}
{(aq^4,q^4/a,q^4,-q^3;q^4)_k} q^k \notag\\
&\qquad\equiv q^{(1-n)/2}[n]\bigg(\frac{-2}{n}\bigg) \pmod{\Phi_n(q)(1-aq^{n})(a-q^{n})} \quad\text{for $n\equiv 1\pmod{4}$},
\end{align*}
where $d=1,4$, and we use $(-1)^{(n-1)/4}=\big(\frac{-2}{n}\big)$ for $n\equiv 1\pmod{4}$.
Applying this $q$-congruence, we can produce a generalization of \eqref{q-f3} with an extra parameter~$a$:
modulo
$$
[n^r]\prod_{j=0}^{(n^{r-1}-1)/d}(1-aq^{(4j+1)n})(a-q^{(4j+1)n}),
$$
we have
\begin{align*}
&\sum_{k=0}^{(n^r-1)/d}(-1)^k[8k+1]\frac{(aq,q/a,q,-q^2;q^4)_k}
{(aq^4,q^4/a,q^4,-q^3;q^4)_k} q^k   \\
&\qquad \equiv q^{(1-n)/2}[n]\bigg(\frac{-2}{n}\bigg)
\sum_{k=0}^{(n^{r-1}-1)/d}(-1)^k[8k+1]_{q^n}  \frac{(aq^n,q^n/a,q^n,-q^{2n};q^{4n})_k}
{(aq^{4n},q^{4n}/a,q^{4n},-q^{3n};q^{2n})_k} q^{nk},
\end{align*}
where $d=1,4$.
\end{proof}

It is easy to see that, when $n=p$ and $q\to 1$, the $q$-supercongruence \eqref{q-f3} reduces to \eqref{f3} when $d=4$ , and
confirms the third supercongruence in \cite[Conjecture 5.3]{Gu19a} when $d=1$. Besides, letting $n=p^2$, $r\mapsto r/2$, and $q\to 1$ in \eqref{q-f3}
we obtain
\begin{align*}
\sum_{k=0}^{(p^r-1)/4}\frac{(8k+1)(\frac{1}{4})_k^3}{k!^3 (-1)^k }
\equiv p^2 \sum_{k=0}^{(p^{r-2}-1)/4} \frac{(8k+1)(\frac{1}{4})_k^3}{k!^3 (-1)^k}  \pmod{p^{2r}}
\end{align*}
for $r\ge 2$ even. This confirms in part the second case of \cite[Conjecture (F.3)]{Sw15}, where the supercongruence is predicted to hold modulo $p^{3r-2}$ for $p\equiv 3\pmod{4}$.

Finally, we should mention the recent work \cite{WY20}, which saw the light after a preliminary version of this work had appeared; there Wang and Yue
gave generalizations of Theorems~\ref{main-4}, ~\ref{main-e} and~\ref{main-f}.

\subsection{Generalizations of Swisher-type supercongruences}
The $m=3$ case of \cite[Conjecture 6.1]{Gu19c} asserts that
\begin{align}
\sum_{k=0}^{(p^r-1)/d}(-1)^k (4k+1)^3\frac{(\frac{1}{2})_k^3}{k!^3}
\equiv p\bigg(\frac{-1}{p}\bigg) \sum_{k=0}^{(p^{r-1}-1)/d}
(-1)^k(4k+1)^3\frac{(\frac{1}{2})_k^3}{k!^3} \pmod{p^{3r-2}},
\label{eq:guo-3}
\end{align}
where $d=1,2$. Here we confirm this supercongruence by establishing its $q$-analogue.
Although there is a $q$-analogue of \eqref{eq:guo-3} modulo $p^3$ for $r=1$ in~\cite{Gu19c}, we need a different one to accomplish the proof of~\eqref{eq:guo-3}.

\begin{lemma}\label{lem:new-1}
Let $n>1$ be an odd integer and $a$ an indeterminate. Then, modulo $\Phi_n(q^2)(1-aq^{2n})(a-q^{2n})$,
\begin{align}
&\sum_{k=0}^{(n-1)/2}(-1)^k [4k+1]_{q^2}[4k+1]^2 \frac{ (aq^2,q^2/a;q^4)_k (q^4;q^8)_k}{ (aq^4,q^4/a;q^4)_k (q^8;q^8)_k}q^{-4k} \notag\\
&\qquad
\equiv q^{1-n}[n]_{q^2}\bigg(\frac{-1}{n}\bigg) \left(1-\frac{(1+q^2)(1-aq^2)(1-q^2/a)}{(1+q^4)(1-q)^2}\right). \label{eq:four-1}
\end{align}
\end{lemma}

\begin{proof}
For $a=q^{-2n}$ or $a=q^{2n}$, the left-hand side of \eqref{eq:four-1} is equal to
\begin{align}
&\sum_{k=0}^{(n-1)/2}(-1)^k [4k+1]_{q^2}[4k+1]^2 \frac{(q^{2-2n},q^{2+2n};q^4)_k (q^4;q^8)_k}
{(q^{4-2n},q^{4+2n};q^4)_k (q^8;q^8)_k}q^{-4k} \notag\\
&\qquad
={} _{8}\phi_{7}\biggl[\begin{matrix}
q^2, \, q^5, \, -q^5, \, q^5, \, q^5, \, -q^2,  \, q^{2+2n}, \, q^{2-2n} \\[1pt]
q, \, -q, \, q, \, q, \, -q^4, \, q^{4-2n}, \, q^{4+2n}
\end{matrix}; q^4,\, -q^{-4} \biggr],
\label{four-2}
\end{align}
where the basic hypergeometric series $_{s+1}\phi_s$ is defined in the introduction.
By Watson's $_8\phi_7$ transformation formula \cite[Appendix (III.18)]{GR04} with $q\mapsto q^4$, $a=q^2$, $b=c=q^5$, $d=-q^2$, $e=q^{2+2n}$ and $n\mapsto(n-1)/2$,
we can write the right-hand side of~\eqref{four-2} as
\begin{align}
&\frac{ (q^6,-q^{2-2n};q^4)_{(n-1)/2}}{(-q^4,q^{4-2n};q^4)_{(n-1)/2}}
\,{}_{4}\phi_{3}\biggl[\begin{matrix}
q^{-4}, \, -q^2, \, q^{2+2n}, \, q^{2-2n} \\[1pt]
q, \, q, \, -q^4
\end{matrix};q^4,\, q^4 \biggr]  \notag\\
&\qquad
=q^{1-n}[n]_{q^2}\bigg(\frac{-1}{n}\bigg)
\left(1-\frac{(1+q^2)(1-q^{2-2n})(1-q^{2+2n})}{(1+q^4)(1-q)^2}\right),  \label{four-3}
\end{align}
which is just the $a=q^{-2n}$ or $a=q^{2n}$ case of the right-hand side of \eqref{eq:four-1}.
This proves that the congruence \eqref{eq:four-1} holds modulo $1-aq^{2n}$ or $a-q^{2n}$.

Moreover, by \cite[Lemma 3.1]{GS20b} it is easy to verify that, for $0\le k\le (n-1)/2$, the $k$-th and $((n-1)/2-k)$-th terms on the left-hand side of \eqref{eq:four-1}
modulo $\Phi_n(q^2)$ cancel each other. Therefore, the left-hand side of \eqref{eq:four-1} is congruent to $0$ modulo
$\Phi_n(q^2)$, and \eqref{eq:four-1} is also true modulo $\Phi_n(q^2)$.
\end{proof}

We are now able to give a complicated $q$-analogue of \eqref{eq:guo-3}.

\begin{theorem}\label{new-1}
Let $n>1$ be an odd integer and $r\ge 2$. Then, modulo
\begin{align*}
\begin{cases}
[n^r]_{q^2}\Phi_n(-q)^2\prod_{j=2}^r\Phi_{n^j}(q^2)^2 &\text{if $n>3$,} \\[5pt]
[n^r]_{q^2}\Phi_n(q^2)\Phi_{n^2}(q^2)\Phi_n(-q)\Phi_{n^2}(-q)\prod_{j=3}^r\Phi_{n^j}(q^2)^2 &\text{if $n=3$,}
\end{cases}
\end{align*}
we have
\begin{align}
&\sum_{k=0}^{(n^r-1)/d}(-1)^k [4k+1]_{q^2}[4k+1]^2 \frac{(q^2;q^4)_k^2 (q^4;q^8)_k }{ (q^4;q^4)_k^2 (q^8;q^8)_k}q^{-4k}  \notag\\
&\qquad \equiv q^{2-2n}[n]_{q^2}\bigg(\frac{-1}{n}\bigg) \frac{(1+q+q^2)(1+q^{4n})}{(1+q^4)(1+q^n+q^{2n})}  \notag\\
&\qquad\quad\times \sum_{k=0}^{(n^{r-1}-1)/d}(-1)^k [4k+1]_{q^{2n}}[4k+1]_{q^n}^2 \frac{(q^{2n};q^{4n})_k^2 (q^{4n};q^{8n})_k }{ (q^{4n};q^{4n})_k^2 (q^{8n};q^{8n})_k}q^{-4nk} , \label{new-1-1}
\end{align}
where $d=1,2$.
\end{theorem}

\begin{proof}[Sketch of proof]
Applying \eqref{eq:four-1}, we can prove the following parametric version of \eqref{new-1-1}:
modulo
\begin{align}
[n^r]_{q^2}\prod_{j=0}^{(n^{r-1}-1)/d}(1-aq^{(4j+2)n})(a-q^{(4j+2)n}),  \label{eq:prod-new0}
\end{align}
we have
\begin{align}
&\sum_{k=0}^{(n^r-1)/d}(-1)^k [4k+1]_{q^2}[4k+1]^2 \frac{ (aq^2,q^2/a;q^4)_k (q^4;q^8)_k}{ (aq^4,q^4/a;q^4)_k (q^8;q^8)_k}q^{-4k}   \notag
\displaybreak[2]\\
&\qquad \equiv q^{1-n}[n]_{q^2}\bigg(\frac{-1}{n}\bigg)
\left(1-\frac{(1+q^2)(1-aq^2)(1-q^2/a)}{(1+q^4)(1-q)^2}\right)  \notag\\
&\qquad\quad\times\left(1-\frac{(1+q^{2n})(1-aq^{2n})(1-q^{2n}/a)}{(1+q^{4n})(1-q^n)^2}\right)^{-1}  \notag\\
&\qquad\quad\times \sum_{k=0}^{(n^{r-1}-1)/d}(-1)^k [4k+1]_{q^{2n}}[4k+1]_{q^n}^2
\frac{ (aq^{2n},q^{2n}/a;q^{4n})_k (q^{4n};q^{8n})_k}{ (aq^{4n},q^{4n}/a;q^{4n})_k (q^{8n};q^{8n})_k}q^{-4nk},  \label{new-1-2}
\end{align}
where $d=1,2$.

Similarly as before, the limit of \eqref{eq:prod-new0} as $a\to 1$ has the factor
\begin{align}
\begin{cases}
\prod_{j=1}^r\Phi_{n^j}(q^2)^{2n^{r-j}+1} &\text{if $d=1$,}\\[5pt]
\prod_{j=1}^r\Phi_{n^j}(q^2)^{n^{r-j}+2} &\text{if $d=2$,}
\end{cases}  \label{eq:new-last}
\end{align}
where one product $\prod_{j=1}^r\Phi_{n^j}(q^2)$ comes from $[n^r]_{q^2}$. However, this time we should be careful of the factor related to $a$ in the common denominator of the two sides of
\eqref{new-1-2}. But it is at most
$$
\big((1+q^{4n})(1-q^n)^2-(1+q^{2n})(1-aq^{2n})(1-q^{2n}/a)\big)(aq^4;q^4)_{(n^r-1)/d} (q^4/a;q^4)_{(n^r-1)/d},
$$
of which the limit as $a\to 1$ only contains the factor
\begin{align*}
\begin{cases}
\Phi_n(q)^2\prod_{j=1}^r\Phi_{n^j}(q^2)^{2n^{r-j}-2} &\text{if $d=1$ and $n>3$,}\\[5pt]
\Phi_n(q)^2\prod_{j=1}^r\Phi_{n^j}(q^2)^{n^{r-j}-1} &\text{if $d=2$ and $n>3$,}\\[5pt]
\Phi_n(q)^2\Phi_{n^2}(q)\prod_{j=1}^r\Phi_{n^j}(q^2)^{2n^{r-j}-2} &\text{if $d=1$ and $n=3$,}\\[5pt]
\Phi_n(q)^2\Phi_{n^2}(q)\prod_{j=1}^r\Phi_{n^j}(q^2)^{n^{r-j}-1} &\text{if $d=2$ and $n=3$,}
\end{cases}
\end{align*}
related to $\Phi_n(q^2),\Phi_{n^2}(q^2),\ldots,\Phi_{n^r}(q^2)$. Here we used the identity
\begin{align}
(1+q^{4n})(1-q^n)^2-(1+q^{2n})(1-q^{2n})^2=-2q^n(1+q^n+q^{2n})(1-q^n)^2.  \label{eq:id}
\end{align}
Thus, letting $a\to 1$ in \eqref{new-1-2}, we see that the $q$-congruence \eqref{new-1-1}
holds modulo
\begin{align*}
\begin{cases}
\Phi_n(q^2)\Phi_n(-q)^2\prod_{j=2}^r\Phi_{n^j}(q^2)^3 &\text{if $n>3$,}\\[5pt]
\Phi_n(q^2)\Phi_{n^2}(q^2)^2\Phi_n(-q)^2\Phi_{n^2}(-q)\prod_{j=3}^r\Phi_{n^j}(q^2)^3 &\text{if $n=3$}.
\end{cases}
\end{align*}

On the other hand, letting $a\to 1$ in \eqref{eq:four-1}, we can easily deduce that the
left-hand side of \eqref{new-1-1} is congruent to
$$
-2q^{2-n}[n]_{q^2}\bigg(\frac{-1}{n}\bigg)\frac{1+q+q^2}{1+q^4} \pmod{\Phi_n(q^2)^3 },
$$
which indicates that it is congruent to 0 modulo $\Phi_n(q)^2$ when $n=3$.
Namely, the $q$-congruence \eqref{new-1-1} holds modulo $\Phi_n(q)^2$ when $n=3$.
Combining this with the previous argument, we conclude that
the $q$-congruence \eqref{new-1-1} is true modulo
\begin{align*}
\begin{cases}
\Phi_n(q^2)\Phi_n(-q)^2\prod_{j=2}^r\Phi_{n^j}(q^2)^3 &\text{if $n>3$,}\\[5pt]
\Phi_n(q^2)^2\Phi_{n^2}(q^2)^2\Phi_n(-q)\Phi_{n^2}(-q)\prod_{j=3}^r\Phi_{n^j}(q^2)^3 &\text{if $n=3$}.
\end{cases}
\end{align*}

Furthermore, based on \eqref{eq:four-1}, along the lines of the proof of \cite[Theorem 1.2]{GZ19a} we can show that
\begin{align}
\sum_{k=0}^{(n-1)/d}(-1)^k [4k+1]_{q^2}[4k+1]^2 \frac{(q^2;q^4)_k^2 (q^4;q^8)_k }{ (q^4;q^4)_k^2 (q^8;q^8)_k}q^{-4k}
\equiv 0\pmod{[n]_{q^2}} \label{new-1-3}
\end{align}
for $d=1,2$. Utilizing this $q$-congruence, we can prove that both sides of \eqref{new-1-1}
are congruent to $0$ modulo $[n^r]_{q^2}$.
\end{proof}

It is not hard to see that, when $n=p$ and $q\to 1$, the $q$-supercongruence \eqref{new-1-1}
reduces to \eqref{eq:guo-3} for $r\ge 2$  (the case $r=1$ of \eqref{eq:guo-3} is obviously true by \cite{Gu19c} or \eqref{new-1-3}).
Moreover, letting $n=p$ and $q\to -1$ in \eqref{new-1-1},
we are led to \eqref{eq:b3-new-1} again.

Similarly, we can partially confirm another conjecture in
\cite{Gu19c}. Recall that the $m=3$ case of \cite[Conjecture 6.2]{Gu19c} may be stated as follows:
\begin{align}
\sum_{k=0}^{(p^r-1)/d} (4k+1)^3\frac{(\frac{1}{2})_k^4}{k!^4} \equiv
p \sum_{k=0}^{(p^{r-1}-1)/d} (4k+1)^3\frac{(\frac{1}{2})_k^4}{k!^4}
\pmod{p^{4r-3}}, \label{eq:guo-4}
\end{align}
where $d=1,2$. Here we prove that \eqref{eq:guo-4} is true
modulo $p^{3r-2}$ using the following $q$-supercongruences.

\begin{theorem}\label{new-1}
Let $n>1$ be an odd integer and $r\ge 2$. Then, modulo
\begin{align*}
[n^r]_{q^2}\Phi_n(-q)^2\prod_{j=2}^r\Phi_{n^j}(q^2)^2,
\end{align*}
we have
\begin{align}
&\sum_{k=0}^{(n^r-1)/d}[4k+1]_{q^2}[4k+1]^2 \frac{(q^2;q^4)_k^4}{(q^4;q^4)_k^4}q^{-4k}\notag\\
&\qquad \equiv q^{2-2n}[n]_{q^2}\frac{1+q^{2n}}{1+q^2}
\sum_{k=0}^{(n^{r-1}-1)/d}[4k+1]_{q^{2n}}[4k+1]_{q^n}^2 \frac{(q^{2n};q^{4n})_k^4}{ (q^{4n};q^{4n})_k^4}q^{-4nk}
, \label{new-2-1}
\end{align}
where $d=1,2$.
\end{theorem}

\begin{proof}[Sketch of proof]
By \cite[Theorem 4.1]{Gu19c}, we have
\begin{align*}
&\sum_{k=0}^{(n-1)/2}[4k+1]_{q^2}[4k+1]^2 \frac{(aq^2,q^2/a;q^4)_k(q^2;q^4)_k^2}{(aq^4,q^4/a;q^4)_k(q^4;q^4)_k^2}q^{-4k} \notag\\
&\qquad\equiv q^{1-n}[n]_{q^2} \left(1-\frac{(1-aq^2)(1-q^2/a)}{(1+q^2)(1-q)^2}\right) \pmod{\Phi_n(q^2)(1-aq^{2n})(a-q^{2n})}.
\end{align*}
Using this $q$-congruence, we can establish the following parametric generalization of \eqref{new-2-1}:
modulo
\begin{align}
[n^r]_{q^2}\prod_{j=0}^{(n^{r-1}-1)/d}(1-aq^{(4j+2)n})(a-q^{(4j+2)n}),  \label{eq:prod-new1}
\end{align}
we have
\begin{align}
&\sum_{k=0}^{(n^r-1)/d}[4k+1]_{q^2}[4k+1]^2 \frac{(aq^2,q^2/a;q^4)_k(q^2;q^4)_k^2}{(aq^4,q^4/a;q^4)_k(q^4;q^4)_k^2}q^{-4k}  \notag\\
&\qquad
\equiv q^{1-n}[n]_{q^2} \left(1-\frac{(1-aq^2)(1-q^2/a)}{(1+q^2)(1-q)^2}\right)\left(1-\frac{(1-aq^{2n})(1-q^{2n}/a)}{(1+q^{2n})(1-q^n)^2}\right)^{-1} \notag\\
&\qquad\quad\times \sum_{k=0}^{(n^{r-1}-1)/d} [4k+1]_{q^{2n}}[4k+1]_{q^n}^2
\frac{ (aq^{2n},q^{2n}/a;q^{4n})_k (q^{2n};q^{4n})_k^2}{(aq^{4n},q^{4n}/a;q^{4n})_k (q^{4n};q^{4n})_k^2}q^{-4nk}.  \label{new-2-2}
\end{align}

Like before, the limit of \eqref{eq:prod-new1} as $a\to 1$ has the factor \eqref{eq:new-last}.
While the factor related to $a$ in the common denominator of the two sides of
\eqref{new-2-2} is at most
$$
\big((1+q^{2n})(1-q^n)^2-(1-aq^{2n})(1-q^{2n}/a)\big)(aq^4;q^4)_{(n^r-1)/d} (q^4/a;q^4)_{(n^r-1)/d},
$$
whose limit as $a\to 1$ only incorporates the factor
\begin{align*}
\begin{cases}
\Phi_n(q)^2\prod_{j=1}^r\Phi_{n^j}(q^2)^{2n^{r-j}-2} &\text{if $d=1$,}\\[5pt]
\Phi_n(q)^2\prod_{j=1}^r\Phi_{n^j}(q^2)^{n^{r-j}-1} &\text{if $d=2$,}
\end{cases}
\end{align*}
related to $\Phi_n(q^2),\Phi_{n^2}(q^2),\ldots,\Phi_{n^r}(q^2)$. Here we utilized the relation
$$
(1+q^{2n})(1-q^n)^2-(1-q^{2n})^2=-2q^n(1-q^n)^2.
$$
Thus, taking the limit of \eqref{new-2-2} as $a\to 1$, we see that the $q$-congruence \eqref{new-2-1}
holds modulo
$\Phi_n(q^2)\Phi_n(-q)^2\prod_{j=2}^r\Phi_{n^j}(q^2)^3$.
Finally, to show that both sides of \eqref{new-2-1}
are also congruent to $0$ modulo $[n^r]_{q^2}$, we only need to use the modulus $[n]_{q^2}$
case of \cite[Theorem 1.4]{Gu19c}.
\end{proof}

It is clear that, when $n=p$ and $q\to 1$, the $q$-supercongruence \eqref{new-2-1}
becomes the modulus $p^{3r-2}$ case of \eqref{eq:guo-4}. Meanwhile, taking $n=p$ and $q\to -1$ in \eqref{new-2-1},
we obtain the modulus $p^{3r}$ case of (C.3) from \cite{Sw15}:
\begin{align*}
\sum_{k=0}^{(p^r-1)/2} (4k+1)\frac{(\frac{1}{2})_k^4}{k!^4}
&\equiv p \sum_{k=0}^{(p^{r-1}-1)/2} (4k+1)\frac{(\frac{1}{2})_k^4}{k!^4} \pmod{p^{3r}}
\end{align*}
and its companion, already proved by the first author in~\cite{Gu20d}.

\subsection{Dwork-type supercongruences involving $(4k-1)$ and $(4k-1)^3$}
The first author \cite[Corollary 5.2]{Gu19c} proved that, for $r\ge 1$,
\begin{align*}
\sum_{k=0}^{(p^r+1)/2} (4k-1)^3\frac{(-\frac{1}{2})_k^3}{k!^3}
&\equiv 3p^r \bigg(\frac{-1}{p^r}\bigg) \pmod{p^{r+2}},  \\
\sum_{k=0}^{p^r-1} (4k-1)^3\frac{(-\frac{1}{2})_k^3}{k!^3}
&\equiv 3p^r\bigg(\frac{-1}{p^r}\bigg) \pmod{p^{r+2}}.
\end{align*}
We observe that these two supercongruences also possess the following Dwork-type generalizations:
\begin{align}
\sum_{k=0}^{(p^r+1)/2} (4k-1)^3\frac{(-\frac{1}{2})_k^3}{k!^3}
&\equiv p\bigg(\frac{-1}{p}\bigg)
\sum_{k=0}^{(p^{r-1}+1)/2} (4k-1)^3\frac{(-\frac{1}{2})_k^3}{k!^3} \pmod{p^{3r-2}},  \label{4k-1-1}\\
\sum_{k=0}^{p^r-1} (4k-1)^3\frac{(-\frac{1}{2})_k^3}{k!^3}
&\equiv p\bigg(\frac{-1}{p}\bigg)
\sum_{k=0}^{p^{r-1}-1} (4k-1)^3\frac{(-\frac{1}{2})_k^3}{k!^3} \pmod{p^{3r-2}}.   \label{4k-1-2}
\end{align}
In fact, these two supercongruences can be further generalized to the $q$-setting. We first give the following result similar
to Lemma \ref{lem:new-1}.

\begin{lemma}\label{lem:new-2}
Let $n>1$ be an odd integer and $a$ an indeterminate. Then, modulo $\Phi_n(q^2)(1-aq^{2n})(a-q^{2n})$,
\begin{align}
&\sum_{k=0}^{(n+1)/2}(-1)^k [4k-1]_{q^2}[4k-1]^2 \frac{ (aq^{-2},q^{-2}/a;q^4)_k (q^{-4};q^8)_k}{(aq^4,q^4/a;q^4)_k (q^8;q^8)_k}q^{4k} \notag\\
&\quad
\equiv \frac{-2q^{-n-3}[n]_{q^2}(1+q^4)}{(1+aq^2)(1+q^2/a)}\bigg(\frac{-1}{n}\bigg)\left(1-\frac{(1+q^2)(1-aq^{-2})(1-q^{-2}/a)}{(1+q^4)(1-q)^2}q^4\right). \label{eq:lem-4k-1}
\end{align}
\end{lemma}

\begin{proof}[Sketch of proof]
For $a=q^{-2n}$ or $a=q^{2n}$, the left-hand side of \eqref{eq:lem-4k-1} can be written as
$$
-q^{-4} \,{} _{8}\phi_{7}\biggl[\begin{matrix}
q^{-2}, \, q^3, \, -q^3, \, q^3, \, q^3, \, -q^{-2}, \, q^{-2+2n}, \, q^{-2-2n} \\[1pt]
q^{-1}, \, -q^{-1}, \, q^{-1}, \, q^{-1}, \, -q^4, \, q^{4-2n}, \, q^{4+2n}
\end{matrix};q^4,\, -q^{4} \biggr]
$$
By Watson's $_8\phi_7$ transformation formula \cite[Appendix (III.18)]{GR04} with
$q\mapsto q^4$, $a= q^{-2}$, $b=c=q^3$, $d=-q^{-2}$, $e=q^{-2+2n}$, and $n\mapsto (n+1)/2$, the above expression is equal to
\begin{align*}
&-q^{-4}\frac{(q^2,-q^{6-2n};q^4)_{(n+1)/2}}
{(-q^4,q^{4-2n};q^4)_{(n+1)/2}}
\,{}_{4}\phi_{3}\biggl[\begin{matrix}
q^{-4}, \, -q^{-2}, \, q^{-2+2n}, \, q^{-2-2n} \\[1pt]
q^{-1}, \, q^{-1}, \, -q^{-4}
\end{matrix}; q^4,\, q^4 \biggr] \\
&\qquad
=\frac{-2q^{n-5}[n]_{q^2}(1+q^4)}{(1+q^{2n-2})(1+q^{2n+2})}\bigg(\frac{-1}{n}\bigg)
\left(1-\frac{(1+q^2)(1-q^{-2+2n})(1-q^{-2-2n})}{(1+q^4)(1-q)^2}q^4\right),
\end{align*}
which is just the $a=q^{-2n}$ or $a=q^{2n}$ case of \eqref{eq:lem-4k-1}.
This means that \eqref{eq:lem-4k-1} is true modulo $(1-aq^{2n})(a-q^{2n})$.
Moreover, in view of \cite[eq.~(5.3)]{Gu19c} with $q\mapsto q^2$,
we can show that \eqref{eq:lem-4k-1} is also true modulo $\Phi_n(q^2)$.
\end{proof}

We are now able to give $q$-analogues of \eqref{4k-1-1} and \eqref{4k-1-2} as follows.

\begin{theorem}
Let $n>1$ be an odd integer and let $r\ge 2$. Then, modulo
\begin{align*}
\begin{cases}
[n^r]_{q^2}\prod_{j=2}^r\Phi_{n^j}(q^2)^2 &\text{if $n>3$,} \\[5pt]
[n^r]_{q^2}\Phi_n(q)\Phi_{n^2}(q^2)\Phi_{n^2}(-q)\prod_{j=3}^r\Phi_{n^j}(q^2)^2 &\text{if $n=3$,}
\end{cases}
\end{align*}
we have
\begin{align}
&\sum_{k=0}^{M_1}(-1)^k [4k-1]_{q^2}[4k-1]^2 \frac{ (q^{-2};q^4)_k^2 (q^{-4};q^8)_k}{(q^4;q^4)_k^2 (q^8;q^8)_k}q^{4k} \notag\\
&\qquad\equiv q^{2n-2}[n]_{q^2}\bigg(\frac{-1}{n}\bigg) \frac{(1+q+q^2)(1+q^{2n})^2}{(1+q^2)^2(1+q^n+q^{2n})} \notag\\
&\qquad\quad\times \sum_{k=0}^{M_2}(-1)^k [4k-1]_{q^{2n}}[4k-1]_{q^n}^2 \frac{ (q^{-2n};q^{4n})_k^2 (q^{-4n};q^{8n})_k}{(q^{4n};q^{4n})_k^2 (q^{8n};q^{8n})_k}q^{4nk}, \label{eq:lem-4k-2}
\end{align}
where $(M_1,M_2)=((n^r+1)/2,(n^{r-1}+1)/2)$ or $(M_1,M_2)=(n^r-1,n^{r-1}-1)$.
\end{theorem}

\begin{proof}[Sketch of proof]
We first consider the case $(M_1,M_2)=((n^r+1)/2,(n^{r-1}+1)/2)$.
Utilizing \eqref{eq:lem-4k-1}, we can prove the following parametric version of \eqref{eq:lem-4k-2}:
modulo
\begin{align}
[n^r]_{q^2}\prod_{j=0}^{(n^{r-1}-1)/2}(1-aq^{(4j+2)n})(a-q^{(4j+2)n}),  \label{eq:prod-new2}
\end{align}
we have
\begin{align}
&\sum_{k=0}^{(n^r+1)/2}(-1)^k [4k-1]_{q^2}[4k-1]^2 \frac{(aq^{-2},q^{-2}/a;q^4)_k (q^{-4};q^8)_k}{(aq^4,q^4/a;q^4)_k (q^8;q^8)_k}q^{4k}   \notag\\
&\qquad \equiv
\frac{q^{3n-3}[n]_{q^2}(1+q^4)}{(1+aq^2)(1+q^2/a)}\bigg(\frac{-1}{n}\bigg)\left(1-\frac{(1+q^2)(1-aq^{-2})(1-q^{-2}/a)}{(1+q^4)(1-q)^2}q^4\right)  \notag\\
&\qquad\quad\times\frac{(1+aq^{2n})(1+q^{2n}/a)}{1+q^{4n}}\left(1-\frac{(1+q^{2n})(1-aq^{-2n})(1-q^{-2n}/a)}{(1+q^{4n})(1-q^n)^2}q^{4n}\right)^{-1}  \notag\\
&\qquad\quad\times \sum_{k=0}^{(n^{r-1}+1)/2}(-1)^k [4k-1]_{q^{2n}}[4k-1]_{q^n}^2
\frac{(aq^{-2n},q^{-2n}/a;q^{4n})_k (q^{-4n};q^{8n})_k}{(aq^{4n},q^{4n}/a;q^{4n})_k (q^{8n};q^{8n})_k}q^{4nk}.  \label{newnew-1-2}
\end{align}

As in the previous considerations, the limit of \eqref{eq:prod-new2} as $a\to 1$ has the factor
$\prod_{j=1}^r\Phi_{n^j}(q^2)^{n^{r-j}+2}$.
This time the factor related to $a$ in the common denominator of the two sides of
\eqref{newnew-1-2} is at most
\begin{align*}
&\big((1+q^{4n})(1-q^n)^2-(1+q^{2n})(1-aq^{-2n})(1-q^{-2n}/a)q^{4n}\big)\\
&\quad\times(aq^4;q^4)_{(n^r+1)/2}(1-aq^{2n(n^{r-1}+1)}) (q^4/a;q^4)_{(n^r+1)/2}(1-q^{2n(n^{r-1}+1)}/a) ,
\end{align*}
whose limit as $a\to 1$ only contains the factor
\begin{align*}
\begin{cases}
\Phi_n(q)^2\Phi_n(q^2)^2\prod_{j=1}^r\Phi_{n^j}(q^2)^{n^{r-j}-1} &\text{if $n>3$,}\\[5pt]
\Phi_n(q)^2\Phi_{n^2}(q)\Phi_n(q^2)^2\prod_{j=1}^r\Phi_{n^j}(q^2)^{n^{r-j}-1} &\text{if $n=3$,}
\end{cases}
\end{align*}
related to $\Phi_n(q^2),\Phi_{n^2}(q^2),\ldots,\Phi_{n^r}(q^2)$.
Here we used the identity \eqref{eq:id} again.
Thus, letting $a\to 1$ in \eqref{newnew-1-2} we find out that the $q$-congruence \eqref{eq:lem-4k-2}
holds modulo
\begin{align*}
\begin{cases}
\Phi_n(-q)\prod_{j=2}^r\Phi_{n^j}(q^2)^3 &\text{if $n>3$,}\\[5pt]
\Phi_{n^2}(q^2)^2\Phi_n(-q)\Phi_{n^2}(-q)\prod_{j=3}^r\Phi_{n^j}(q^2)^3 &\text{if $n=3$}.
\end{cases}
\end{align*}

On the other hand, letting $a\to 1$ in \eqref{eq:lem-4k-1} we can easily deduce that the
left-hand side of \eqref{eq:lem-4k-2} is congruent to
$$
4q^{-n-2}[n]_{q^2}\bigg(\frac{-1}{n}\bigg)\frac{1+q+q^2}{(1+q^2)^2} \pmod{\Phi_n(q^2)^3 },
$$
which indicates that it is congruent to 0 modulo $\Phi_n(q)^2$ when $n=3$, and so \eqref{eq:lem-4k-2}
is true modulo $\Phi_n(q)^2$ when $n=3$.
From this we immediately deduce that
the $q$-congruence \eqref{eq:lem-4k-2} is true modulo
\begin{align*}
\begin{cases}
\Phi_n(q^2)\prod_{j=2}^r\Phi_{n^j}(q^2)^3 &\text{if $n>3$,}\\[5pt]
\Phi_n(q)^2\Phi_{n^2}(q^2)^2\Phi_n(-q)\Phi_{n^2}(-q)\prod_{j=3}^r\Phi_{n^j}(q^2)^3 &\text{if $n=3$}.
\end{cases}
\end{align*}

Furthermore, based on \eqref{eq:lem-4k-1}, along the lines of the proof of \cite[Theorem 1.2]{GZ19a} we can show that
\begin{align*}
\sum_{k=0}^{(n+1)/2}(-1)^k [4k-1]_{q^2}[4k-1]^2 \frac{ (q^{-2};q^4)_k^2 (q^{-4};q^8)_k}{(q^4;q^4)_k^2 (q^8;q^8)_k}q^{4k}
\equiv 0\pmod{[n]_{q^2}}.
\end{align*}
With the help of this $q$-congruence, we deduce that both sides of \eqref{eq:lem-4k-2}
are congruent to $0$ modulo $[n^r]_{q^2}$. This proves \eqref{eq:lem-4k-2} for $(M_1,M_2)=((n^r+1)/2,(n^{r-1}+1)/2)$.

For $(M_1,M_2)=(n^r-1,n^{r-1}-1)$, the proof follows from the same argument. In this case
the corresponding parametric generalization holds modulo
$$
[n^r]_{q^2}\prod_{j=0}^{n^{r-1}-2}(1-aq^{(4j+2)n})(a-q^{(4j+2)n}).
$$
At the same time, the factor related to $a$ in the common denominator of the two sides is at most
\begin{align*}
&\big((1+q^{4n})(1-q^n)^2-(1+q^{2n})(1-aq^{-2n})(1-q^{-2n}/a)q^{4n}\big)\\
&\quad\times(aq^4;q^4)_{n^r-1} (q^4/a;q^4)_{n^r-1}.
\end{align*}
Therefore, we are led to the same modulus when we take the limit as $a\to 1$.
\end{proof}

It is not hard to see that \eqref{4k-1-1} and \eqref{4k-1-2} follow from \eqref{eq:lem-4k-2} by taking
$n=p$ and $q\to 1$. In addition, we obtain the following supercongruences by setting $n=p$ and $q\to -1$ in \eqref{eq:lem-4k-2}:
\begin{align*}
\sum_{k=0}^{(p^r+1)/2} (4k-1)\frac{(-\frac{1}{2})_k^3}{k!^3}
&\equiv p\bigg(\frac{-1}{p}\bigg)
\sum_{k=0}^{(p^{r-1}+1)/2} (4k-1)\frac{(-\frac{1}{2})_k^3}{k!^3} \pmod{p^{3r-2}},
\displaybreak[2]\\
\sum_{k=0}^{p^r-1} (4k-1)\frac{(-\frac{1}{2})_k^3}{k!^3}
&\equiv p\bigg(\frac{-1}{p}\bigg)
\sum_{k=0}^{p^{r-1}-1} (4k-1)\frac{(-\frac{1}{2})_k^3}{k!^3} \pmod{p^{3r-2}},
\end{align*}
which are related to the supercongruences in \cite[Corollary 5.3]{Gu19c}.

\subsection{Generalizations of Rodriguez-Villegas' supercongruences}
Mortenson \cite{Mo03a,Mo03b} proved  the following four
supercongruences conjectured by Rodriguez-Villegas\cite[eq.~(36)]{RV03}:
\begin{align}
\sum_{k=0}^{p-1}\frac{1}{16^k}{2k\choose k}^2
&\equiv \bigg(\frac{-1}{p}\bigg)\pmod{p^2} \quad\text{for $p>2$}, \label{eq:RV1}
\displaybreak[2]\\
\sum_{k=0}^{p-1}\frac{1}{27^k}{3k\choose 2k}{2k\choose k}
&\equiv \bigg(\frac{-3}{p}\bigg)\pmod{p^2} \quad\text{for $p>3$}, \label{eq:RV2}
\displaybreak[2]\\
\sum_{k=0}^{p-1}\frac{1}{64^k}{4k\choose 2k}{2k\choose k}
&\equiv \bigg(\frac{-2}{p}\bigg)\pmod{p^2} \quad\text{for $p>2$}, \label{eq:RV3}
\displaybreak[2]\\
\sum_{k=0}^{p-1}\frac{1}{432^k}{6k\choose 3k}{3k\choose k}
&\equiv \bigg(\frac{-1}{p}\bigg)\pmod{p^2} \quad\text{for $p>3$}.\label{eq:RV4}
\end{align}
For an elementary proof of \eqref{eq:RV1}--\eqref{eq:RV4}, we refer the reader to \cite{Su14};
for a recent generalization of them, see \cite{Li17}.
Some $q$-analogues of \eqref{eq:RV1}--\eqref{eq:RV4} can be found in \cite{GZ14,GPZ17,NP18,Gu19d}.
In particular, the first author \cite[Corollary 1.4]{Gu19d} proved that, for positive integers $m$, $n$ and $s$ with $\gcd(m,n)=1$,
we have
\begin{equation}
\sum_{k=0}^{n-1}\frac{2(q^s,q^{m-s};q^m)_k q^{mk}}{(q^m;q^m)_{k}^2 (1+q^{mk})}
\equiv (-1)^{\langle -s/m\rangle_n} \pmod{\Phi_n(q)^2},  \label{q-rv}
\end{equation}
where $\langle x\rangle_n$ denotes the least nonnegative residue of $x$ modulo $n$.

Here we give a Dwork-type generalization of \eqref{q-rv} for $m=2$ and $s=1$.

\begin{theorem}
Let $n>1$ be an odd integer and let $r\ge 1$. Then, modulo $\prod_{j=1}^r\Phi_{n^j}(q)^2$,
\begin{equation}
\sum_{k=0}^{(n^r-1)/d}\frac{2(q;q^2)_k^2 q^{2k}}{(q^2;q^2)_{k}^2 (1+q^{2k})}
\equiv \bigg(\frac{-1}{n}\bigg)\sum_{k=0}^{(n^{r-1}-1)/d}\frac{2(q^n;q^{2n})_k^2 q^{2nk}}{(q^{2n};q^{2n})_{k}^2 (1+q^{2nk})},  \label{q-rvv}
\end{equation}
where $d=1,2$.
\end{theorem}

\begin{proof}[Sketch of proof]
By \cite[Corollary 1.4]{Gu19d}, we have
\begin{equation*}
\sum_{k=0}^{(n-1)/2}\frac{2(aq,q/a;q^2)_k q^{2k}}{(q^2;q^2)_{k}^2 (1+q^{2k})}
\equiv \bigg(\frac{-1}{n}\bigg) \pmod{(1-aq^n)(a-q^n)}.
\end{equation*}
This enables us to establish the following parametric generalization of \eqref{q-rv}:
modulo
$$
\prod_{j=0}^{(n^{r-1}-1)/d}(1-aq^{(2j+1)n})(a-q^{(2j+1)n}),
$$
we have
\begin{equation*}
\sum_{k=0}^{(n^r-1)/d}\frac{2(aq,q/a;q^2)_k q^{2k}}{(q^2;q^2)_{k}^2 (1+q^{2k})}
\equiv \bigg(\frac{-1}{n}\bigg)
\sum_{k=0}^{(n^{r-1}-1)/d}\frac{2(aq^n,q^n/a;q^{2n})_k q^{2nk}}{(q^{2n};q^{2n})_{k}^2 (1+q^{2nk})}.
\qedhere
\end{equation*}
\end{proof}

Letting $n=p$ and $q\to 1$ in \eqref{q-rvv} we obtain the following Dwork-type supercongruence:
\begin{align}
\sum_{k=0}^{(p^r-1)/d}\frac{1}{16^{k}}{2k\choose k}^2
\equiv  \bigg(\frac{-1}{p}\bigg)\sum_{k=0}^{(p^{r-1}-1)/d}\frac{1}{16^{k}}{2k\choose k}^2 \pmod{p^{2r}}, \label{eq:rv-1}
\end{align}
where $d=1,2$. This confirms, for the first time, predictions of Roberts and Rodriguez-Villegas from \cite{RRV19}.

Numerical calculation suggests that \eqref{eq:RV2}--\eqref{eq:RV4} have similar
generalizations modulo $p^{2r}$. It seems that these supercongruences even have neat $q$-analogues as follows.

\begin{conjecture}
\label{conj-RV}
Let $m$ and $s$ be positive integers with $s<m$. Let $n>1$ be an odd integer with $n\equiv \pm 1\pmod{m}$. Then, for $r\ge 2$, modulo $\prod_{j=1}^r\Phi_{n^j}(q)^2$,
\begin{equation}
\sum_{k=0}^{n^r-1}\frac{2(q^s,q^{m-s};q^m)_k q^{mk}}{(q^m;q^m)_{k}^2 (1+q^{mk})}
\equiv (-1)^{\langle -s/m\rangle_n}\sum_{k=0}^{n^{r-1}-1}\frac{2(q^{sn},q^{mn-sn};q^{mn})_k q^{mnk}}{(q^{mn};q^{mn})_{k}^2 (1+q^{mnk})}.  \label{q-rv-conj}
\end{equation}
\end{conjecture}

Note that \eqref{q-rvv} with $d=1$ is just the $(m,s)=(2,1)$ case of \eqref{q-rv-conj}. Although there is a parametric generalization of \eqref{q-rv-conj} for $r=1$ (see \cite[Corollary 1.4]{Gu19d}),
we are not aware of a parametric extension for $r\ge 2$.
After appearance of preliminary version of this paper, Ni \cite{Ni20} managed to prove the $n\equiv1\pmod m$ case of Conjecture~\ref{conj-RV} using the method of creative microscoping.
However, we believe that the remaining $n\equiv-1\pmod m$ case should still be very difficult.

\section{Open problems and concluding remarks}
\label{sec4}

\subsection{Open problems}
\label{sec4.1}
First we give some related open problems for further study. Recall that
Swisher's conjectural supercongruence (A.3) for $p\equiv 1\pmod{4}$ can be stated as follows:
\begin{align}
\sum_{k=0}^{(p^r-1)/2} (-1)^{k}(4k+1)\frac{(\frac{1}{2})_k^5}{k!^5}
\equiv
 -p\Gamma_p(1/4)^4 \sum_{k=0}^{(p^{r-1}-1)/2} (-1)^{k}(4k+1)\frac{(\frac{1}{2})_k^5}{k!^5}\pmod{p^{5r}}, \label{eq:a3}
\end{align}
where $\Gamma_p(x)$ denotes the $p$-adic gamma function and $p>5$.
Swisher \cite{Sw15} proves herself \eqref{eq:a3} for $r=1$.
We find the following partial $q$-analogue of \eqref{eq:a3}.

\begin{conjecture}\label{conj:1}
Let $n>1$ be an integer with $n\equiv 1\pmod{4}$ and let $r\ge 1$. Then, modulo $[n^r]\prod_{j=1}^r\Phi_{n^j}(q)^2$,
\begin{align}
&\sum_{k=0}^{(n^r-1)/d}
(-1)^k[4k+1]\frac{(q;q^2)_k^4(q^2;q^{4})_k} {(q^2;q^2)_k^4(q^4;q^4)_k}q^k  \notag\\
&\qquad\equiv \frac{(q^2;q^4)_{(n^r-1)/4}^2 (q^{4n};q^{4n})_{(n^{r-1}-1)/4}^2}{(q^4;q^4)_{(n^r-1)/4}^2 (q^{2n};q^{4n})_{(n^{r-1}-1)/4}^2}[n]
\notag\\
&\qquad\quad\times
\sum_{k=0}^{(n^{r-1}-1)/d}
(-1)^k[4k+1]_{q^n}\frac{(q^n;q^{2n})_k^4(q^{2n};q^{4n})_k} {(q^{2n};q^{2n})_k^4(q^{4n};q^{4n})_k}q^{nk}. \label{q-a3}
\end{align}
\end{conjecture}

Note that the case $r=1$ of \eqref{q-a3} has been proved by the first author \cite{Gu20c}.
Therefore, the left-hand side of \eqref{eq:a3} is congruent to 0 modulo $p^{r}$ (including $p=5$).
To see \eqref{q-a3} is indeed a $q$-analogue of \eqref{eq:a3} modulo $p^{3r}$, one needs to check that
$$
\frac{(\frac{1}{2})_{(p^r-1)/4}^2(1)_{(p^{r-1}-1)/4}^2 }{(1)_{(p^r-1)/4}^2(\frac{1}{2})_{(p^{r-1}-1)/4}^2}
\equiv -\Gamma_p(1/4)^4  \pmod{p^{2r}}
$$
for any prime $p\equiv 1\pmod{4}$. This is similar to the case $r=1$ treated by Van Hamme in \cite[Theorem 3]{VH86}.

We also have the following complete $q$-analogues of \eqref{eq:b3-new-1} and \eqref{eq:b3-new-2}.
\begin{conjecture}\label{conj:2}
Let $n>1$ be an odd integer and let $r\geqslant
1$. Then, modulo $[n^r]\prod_{j=1}^r\Phi_{n^j}(q)^2$,
\begin{align}
\sum_{k=0}^{(n^r-1)/d}(-1)^k[4k+1]\frac{(q^2;q^4)_k^3}{(q^4;q^4)_k^3}\,q^{k}
&\equiv\dfrac{[n]_{q^2}(-q^3;q^4)_{(n^r-1)/2}(-q^{5n};q^{4n})_{(n^{r-1}-1)/2}}
{(-q^5;q^4)_{(n^r-1)/2} (-q^{3n};q^{4n})_{(n^{r-1}-1)/2}}\,(-q)^{(1-n)/2} \notag\\
&\quad\times
\sum_{k=0}^{(n^{r-1}-1)/d}(-1)^k[4k+1]_{q^n}\frac{(q^{2n};q^{4n})_k^3}{(q^{4n};q^{4n})_k^3}\,q^{nk}, \label{q-c3}
\end{align}
where $d=1,2$.
\end{conjecture}

Note that the case $r=1$ of \eqref{q-c3} was proved by the authors in \cite{GZ20}. However, using the creative microscoping method in a usual manner, we
cannot prove Conjectures \ref{conj:1} and \ref{conj:2} for $r>1$ in general.

Based on \cite[Theorem 1.1]{GZ20} we formulate a partial $q$-analogue of Swisher's (H.3) supercongruence \cite{Sw15}.

\begin{conjecture}\label{conj:3}
Let $n>1$ be an integer with $n\equiv 1\pmod{4}$ and let $r\ge 1$. Then, modulo $\prod_{j=1}^r\Phi_{n^j}(q)^2$,
\begin{align*}
\sum_{k=0}^{(n^r-1)/d}
\frac{(1+q^{4k+1})\,(q^2;q^4)_k^3}{(1+q)\,(q^4;q^4)_k^3}\,q^{k}  &\equiv \dfrac{[n]_{q^2}(q^3;q^4)_{(n^r-1)/2}(q^{5n};q^{4n})_{(n^{r-1}-1)/2}}
{(q^5;q^4)_{(n^r-1)/2} (q^{3n};q^{4n})_{(n^{r-1}-1)/2}}\,q^{(1-n)/2}  \\
&\quad\times\sum_{k=0}^{(n^{r-1}-1)/d}
\frac{(1+q^{(4k+1)n})\,(q^{2n};q^{4n})_k^3}{(1+q^n)\,(q^{4n};q^{4n})_k^3}\,q^{nk}, 
\end{align*}
where $d=1,2$.
\end{conjecture}

We also have the following partial $q$-analogues of \eqref{eq:last-1} and \eqref{eq:last-2}.

\begin{conjecture}\label{conj:4}
Let $n>1$ be an odd integer and let $r\ge
1$. Then, modulo $[n^r]\Phi_{n^r}(q)\prod_{j=1}^r\Phi_{n^j}(q)$,
\begin{align}
&\sum_{k=0}^{(n^r-1)/d}(-1)^k[3k+1]\frac{(q;q^2)_k^3}{(q;q)_k^3}  \notag\\
&\qquad \equiv q^{((n^r-1)^2-n(n^{r-1}-1)^2)/4}[n]\bigg(\frac{-1}{n}\bigg)
\sum_{k=0}^{(n^{r-1}-1)/d}(-1)^k[3k+1]_{q^n}\frac{(q^{n};q^{2n})_k^3}{(q^{n};q^{n})_k^3}, \label{eq:div-2-2}
\end{align}
where $d=1,2$.
\end{conjecture}

We point out that the case $r=d=1$ of \eqref{eq:div-2-2} was established by the first author in \cite{Gu20a},
while the case $r=1$, $d=2$ of \eqref{eq:div-2-2} was confirmed by the authors in \cite{GZ19a}.

Similarly, we have the following partial $q$-analogues of \eqref{eq:b3-new-1} and \eqref{eq:b3-new-2}.
The proof of the case $r=1$ can be found in \cite{Gu18b,GZ19a}.

\begin{conjecture}\label{conj:5}
Let $n>1$ be an odd integer and let $r\ge
1$. Then, modulo $[n^r]\Phi_{n^r}(q)\prod_{j=1}^r\Phi_{n^j}(q)$,
\begin{align*}
&\sum_{k=0}^{(n^r-1)/d}(-1)^k[4k+1]\frac{(q;q^2)_k^3}{(q^2;q^2)_k^3}\,q^{k^2}  \notag\\
&\qquad \equiv q^{((n^r-1)^2-n(n^{r-1}-1)^2)/4}[n]\bigg(\frac{-1}{n}\bigg)
\sum_{k=0}^{(n^{r-1}-1)/d}(-1)^k[4k+1]_{q^n}\frac{(q^{n};q^{2n})_k^3}{(q^{2n};q^{2n})_k^3}\,q^{nk^2},
\end{align*}
where $d=1,2$.
\end{conjecture}

We also have a $q$-analogue of \eqref{eq:rv-1} modulo $p^{r+1}$, which seems difficult to prove;
for the case $r=1$, see \cite{GZ14}.

\begin{conjecture}\label{conj:6}
Let $n>1$ be an odd integer and let $r\ge 1$. Then, modulo $\Phi_{n^r}(q)\prod_{j=1}^r\Phi_{n^j}(q)$,
\begin{equation*}
\sum_{k=0}^{(n^r-1)/d}\frac{(q;q^2)_k^2 }{(q^2;q^2)_{k}^2}
\equiv q^{(1-n)(1+n^{2r-1})/4}\bigg(\frac{-1}{n}\bigg)
\sum_{k=0}^{(n^{r-1}-1)/d}\frac{(q^n;q^{2n})_k^2}{(q^{2n};q^{2n})_{k}^2 },
\end{equation*}
where $d=1,2$.
\end{conjecture}

The authors \cite[Theorem 4.14]{GZ19a} utilized Andrews' $q$-analogue of Gauss' $_2F_1(-1)$ sum (see \cite[Appendix (II.11)]{GR04})
to prove that, for $n\equiv 3\pmod{4}$,
$$
\sum_{k=0}^{(n-1)/2}\frac{(q;q^2)_k^2 }{(q^2;q^2)_k(q^4;q^4)_k}q^{2k} \equiv 0\pmod{\Phi_n(q)^2}.
$$
Using the same method, we can show that, for $n\equiv 1\pmod{4}$,
$$
\sum_{k=0}^{(n-1)/2}\frac{(q;q^2)_k^2 }{(q^2;q^2)_k(q^4;q^4)_k} q^{2k}
\equiv \bigg(\frac{-2}{n}\bigg)q^{(n-1)(n+3)/8}\frac{(q^2;q^4)_{(n-1)/4}}{(q^4;q^4)_{(n-1)/4}}\pmod{\Phi_n(q)^2}.
$$

We have the following Dwork-type generalizations of the above $q$-congruence.
\begin{conjecture}Let $n>1$ be an integer with $n\equiv 1\pmod{4}$ and let $r\geqslant 1$. Then, modulo $\Phi_{n^r}(q)\prod_{j=1}^r\Phi_{n^j}(q)$,
\begin{align*}
&\sum_{k=0}^{(n^r-1)/d}\frac{(q;q^2)_k^2 }{(q^2;q^2)_k(q^4;q^4)_k}q^{2k} \\
&\quad\equiv \bigg(\frac{-2}{n}\bigg)q^{((n^r-1)(n^r+3)-n(n^{r-1}-1)(n^{r-1}+3))/8}\frac{(q^2;q^4)_{(n^r-1)/4}(q^{4n};q^{4n})_{(n^{r-1}-1)/4}}{(q^4;q^4)_{(n^r-1)/4}(q^{2n};q^{4n})_{(n^{r-1}-1)/4}}\\
&\quad\quad\times \sum_{k=0}^{(n^{r-1}-1)/d}\frac{(q^n;q^{2n})_k^2 }{(q^{2n};q^{2n})_k(q^{4n};q^{4n})_k} q^{2nk},
\end{align*}
where $d=1,2$.
\end{conjecture}

For the case where $n$ is a prime and $q$ tends to $1$, the following stronger Dwork-type supercongruences seem to be true: for any prime $p\equiv 1\pmod{4}$ and $d=1,2$,
\begin{align*}
\sum_{k=0}^{(p^r-1)/d}\frac{1}{32^k}{2k\choose k}^2
\equiv \bigg(\frac{-2}{p}\bigg)\frac{(\frac{1}{2})_{(p^r-1)/4}(1)_{(p^{r-1}-1)/4}}{(1)_{(p^r-1)/4}(\frac{1}{2})_{(p^{r-1}-1)/4}}
\sum_{k=0}^{(p^{r-1}-1)/d}\frac{1}{32^k}{2k\choose k}^2 \pmod{p^{2r}}.
\end{align*}
Note that the $r=1$ case was first proved by Sun \cite{SunZH}.

Recently, the first author \cite{Gu19b} proved the $q$-congruence
\begin{align}
\sum_{k=0}^{n-1}\frac{q^k}{(-q;q)_{k}}{2k\brack k}_q\equiv \bigg(\frac{-1}{n}\bigg)q^{(n^2-1)/4} \pmod{\Phi_n(q)^2}, \label{eq:Tauraso-1}
\end{align}
conjectured earlier by Tauraso \cite{Ta13} for $n$ an odd prime.
The first author also conjectured that
$$
\sum_{k=0}^{n-1} q^k {2k\brack k} \equiv \left(\frac{-3}{n}\right)q^{(n^2-1)/3} \pmod{\Phi_n(q)^2},
$$
which was confirmed by Liu and Petrov \cite{LP19}.
We indicate the following Dwork-type $q$-generalizations of them.

\begin{conjecture}\label{conj:7}
Let $n>1$ be an odd integer and let $r\ge 1$. Then, modulo $\Phi_{n^r}(q)^{2-d}\prod_{j=1}^r\Phi_{n^j}(q)$,
\begin{align*}
\sum_{k=0}^{(n^r-1)/d}\frac{q^k}{(-q;q)_{k}}{2k\brack k}_q
&\equiv q^{(n-1)(1+n^{2r-1})/4}\bigg(\frac{-1}{n}\bigg)
\sum_{k=0}^{(n^{r-1}-1)/d}\frac{q^{nk}}{(-q^n;q^n)_{k}}{2k\brack k}_{q^n}, \\
\sum_{k=0}^{(n^r-1)/d}q^k {2k\brack k}_q
&\equiv q^{(n-1)(1+n^{2r-1})/3}\bigg(\frac{-3}{n}\bigg)
\sum_{k=0}^{(n^{r-1}-1)/d}q^{nk}{2k\brack k}_{q^n},
\end{align*}
where $d=1,2$. When $d=1$, the second $q$-congruence still holds for even integers $n$.
\end{conjecture}

Sun \cite[Conjecture 3\,(ii),(iii)]{Su19} conjectured that
\begin{align}
\sum_{k=0}^{p^r-1}\frac{1}{2^k}{2k\choose k}
&\equiv  \bigg(\frac{-1}{p}\bigg)\sum_{k=0}^{p^{r-1}-1}\frac{1}{2^k}{2k\choose k} \pmod{p^{2r}}
\quad \text{for $p>2$},  \label{eq:sun-1} \\
\sum_{k=0}^{p^r-1}{2k\choose k}
&\equiv  \bigg(\frac{-3}{p}\bigg)\sum_{k=0}^{p^{r-1}-1}{2k\choose k} \pmod{p^{2r}}, \label{eq:sun-2}
\end{align}
and these expectations were recently confirmed by Zhang and Pan in~\cite{ZP20}.
The supercongruences \eqref{eq:sun-1} and \eqref{eq:sun-2} are somewhat different from the other ones discussed in this paper,
because already for $r=1$ they are valid for the truncations at $p-1$ but not at $(p-1)/2$.
Apart from what is stated in Conjecture~\ref{conj:7}, we could not succeed in finding complete $q$-analogues for the pair of supercongruences.

Although the method of creative microscoping\,---\,in particular, its version developed in this paper\,---\,is an adequate tool in dealing with the congruences conjectured above,
the difficulty of finding appropriate parametric $q$-congruences and $q$-hypergeometric sums seems to be a principal obstacle.
The underlying identities require a human touch, and this fact makes it impossible to predict when resolutions of (some of these) conjectures take place.

\subsection{Dwork-type $q$-congruences}
\label{sec4.2}

Dwork-type (super)congruences \eqref{eq2} we address in this paper all correspond to the choice $z=1$ and a specific shape of the unit root $\omega(z)$, namely, associated with a Dirichlet quadratic character.
Nevertheless, there is experimental evidence for existence of $q$-congruences of the type
\begin{equation}
\sum_{k=0}^{(n^r-1)/d}A_k(q)
\equiv\omega(q)\sum_{k=0}^{(n^{r-1}-1)/d}A_k(q^n)
\label{q-Dwork}
\end{equation}
modulo $\prod_{j=1}^r\Phi_{n^j}(q)$, say, for a suitable choice of $q$-hypergeometric term $A_k(q)$,
in which the `$q$-unit root' $\omega(q)$ has a more sophisticated structure than just $q^N\bigl(\frac{-D}n\bigr)$.
One such example for truncations of the $q$-series
$$
\sum_{k=0}^\infty\frac{(q;q^2)_k^4}{(q^2;q^2)_k^4}\,q^{2k}
$$
is suggested by Conjectures~4.1--4.3 in~\cite{Gu20d}, though an explicit form of $\omega(q)$ remains unclear.
A significance of this particular example is due to the connection of its $q\to1$ limit with the Dwork-type supercongruence
$$
\sum_{k=0}^{p^r-1}\frac{(\frac12)_k^4}{k!^4}
\equiv\omega_p\sum_{k=0}^{p^{r-1}-1}\frac{(\frac12)_k^4}{k!^4}\pmod{p^{3r}}
\quad\text{for}\; p>2, \; r=1,2,\dots,
$$
conjectured in \cite{RRV19}, with $r=1$ instance established earlier by Kilbourn~\cite{Ki06} (see also~\cite{LTYZ17}).
Here the unit root $\omega_p$ is the $p$-adic zero, not divisible by~$p$, of quadratic polynomial $T^2-a(p)T+p^3$,
where the traces of Frobenius $a(p)$ originate from the modular form $\sum_{m=1}^\infty a(m)q^m=q\,(q^2;q^2)_\infty^4(q^4;q^4)_\infty^4$.
The congruence is remarkably related to a modular Calabi--Yau threefold \cite{AO00},
and we expect that its $q$-analogue will shed light on a $q$-deformation of the modular form and of the cohomology groups of the threefold \cite{Sc17}.

It is certain that $q$-congruences of the type \eqref{q-Dwork} not only provide us with an efficient method for proving their $q\to1$ specializations but also have their own right to exist.

\medskip
\noindent
\textbf{Acknowledgements.}
We thank the two anonymous referees for their valuable and enthusiastic feedback.
The first author also thanks Zhi-Wei Sun for helpful comments on \eqref{eq:sun-1} and \eqref{eq:sun-2}.

\end{document}